\IfFileExists{./prepreamble-amspreprint.sty}{\RequirePackage[packages,theorems]{prepreamble-amspreprint}}{}

\documentclass[english]{amsart}

\RequirePackage{biblatex}
\RequirePackage{ltxcmds}
\IfFileExists{./preamble-amspreprint.sty}{\RequirePackage[packages,theorems]{preamble-amspreprint}}{}

\usepackage{manuscript}
\usepackage{KMS-local}
\usepackage{KMS-standard-packages}
\usepackage{scoop-semantic}
\usepackage{scoop-style}
\usepackage{thmtools}

\IfFileExists{./postpreamble-amspreprint.sty}{\RequirePackage[packages,theorems]{postpreamble-amspreprint}}{}

\makeatletter
\@ifpackageloaded{hyperref}{%
	\hypersetup{
		pdftitle = {Admissible and attainable convergence behavior with stagnation mirroring in restarted (block) GMRES},
		pdfauthor = {Kirk M Soodhalter},
		pdfkeywords = {Krylov subspace methods, GMRES, Restarting, Block Krylov subspaces, Admissible convergence behavior},
		pdfcreator = {Created using the Scoop Template Engine version 1.6.0.}
	}
}{
	\pdfinfo{
		/Title (Admissible and attainable convergence behavior with stagnation mirroring in restarted (block) GMRES)
		/Author (Kirk M Soodhalter)
		/Subject ()
		/Keywords (Krylov subspace methods, GMRES, Restarting, Block Krylov subspaces, Admissible convergence behavior)
		/Creator (Created using the Scoop Template Engine version 1.6.0.)
	}
}
\makeatother

\title[Prescribing convergence for restarted Block GMRES]{Admissible and attainable convergence behavior with stagnation mirroring in restarted (block) GMRES}

\author[K. M. Soodhalter]{Kirk M Soodhalter}
\address[K. M. Soodhalter]{School of Mathematics, Trinity College Dublin, College Green, Dublin 2, Ireland}
\email{\detokenize{ksoodha@maths.tcd.ie}}
\urladdr{https://math.soodhalter.com}

\thanks{The author's work has been partially supported by Research Ireland grant 22/EPSRC/3857.}

\date{\today}

\dedicatory{}

\begin{document}

\begin{abstract}
In this work, we describe how to construct matrices and block
right-hand sides that exhibit a specified restarted block \gmres convergence
pattern, such that the eigenvalues and Ritz values at each iteration can be
chosen independent of the specified convergence behavior.  This work is a
generalization of the work in [Meurant and Tebbens, Num.~Alg.~2019] in which
the authors do the same for restarted non-block \gmres.  We use the same tools
as were used in [Kub\'inov\'a and Soodhalter, SIMAX 2020], namely to analyze
block \gmres as an iteration over a right vector space with scalars from the
$^\ast$-algebra of matrices.  To facilitate our work, we also extend the work
of Meurant and Tebbens and offer alternative proofs of some of their
results, that can be more easily generalized to the block setting. 

\end{abstract}


\makeatletter
\ltx@ifpackageloaded{hyperref}{%
\subjclass[2010]{\href{https://mathscinet.ams.org/msc/msc2020.html?t=65F10}{65F10}, \href{https://mathscinet.ams.org/msc/msc2020.html?t=65N12}{65N12}, \href{https://mathscinet.ams.org/msc/msc2020.html?t=15B57}{15B57}, \href{https://mathscinet.ams.org/msc/msc2020.html?t=45B05}{45B05}, \href{https://mathscinet.ams.org/msc/msc2020.html?t=45A05}{45A05}}
}{%
}
\makeatother

\maketitle

\section{Introduction}\label{section:introduction}
We consider the solution of 
\begin{align}
	A\bX
	=
	\bB,
	\label{eqn:AXB}
\end{align}
where $A\in\C^{n\times n}$ is a general non-Hermitian matrix, and $\bB\in\C^{n\times p}$ with $p>1$, using a block \gmres
iteration (\blgmres), with a focus on the theoretical properties of a restarted version of \blgmres.  In particular, we extend our
work on characterizing admissible and attainable convergence behavior of \blgmres \cite{KubinovaSoodhalter:2020:1} to the case of
restarted \blgmres. For non-block ($p=1$) restarted \gmres, such a characterization is derived in
\cite{DuintjerTebbensMeurant:2019:1}. To facilitate extension of this work to the restarted \blgmres case, we provide some
simplified versions of proofs in \cite{DuintjerTebbensMeurant:2019:1}. 

\subsubsection*{Non-block notation}
For non-block \gmres discussed in
\Cref{section:admissible-convergence-GMRES,section:admissible-convergence-rGMRES},
we adopt the convention that matrices are denoted by unstyled capital letters, \eg $H$.
Vectors are denoted by bold, lower-case letters, \eg $\bu$.  Scalars are denoted by
unstyled lower-case letters or by lower-case Greek letters, \eg $p$ or $\gamma$.

\subsubsection*{Block notation}
For \blgmres, we follow the approach in \cite{KubinovaSoodhalter:2020:1} and
consider \blgmres as an iteration of matrices and vectors with block entries
coming from a $^\ast$-algebra of $p\times p$ square matrices, \cf
\Cref{section:admissible-convergence-blGMRES}.  We carry over the matrix,
single vector, and scalar notations from the non-block \gmres case. We further
denote multi-column $n\times p$ block vectors (understood as vectors whose
entries are from the $^\ast$-algebra) with bolded capital letters (\eg $\bV$)
and matrices with entries that are $^\ast$-algebra with caligraphic capital
letters, \eg $\cH$.  Matrices whose columns have been generated in terms of
block vectors (denoting basis matrices) are denoted by bold caligraphic capital
letters, \eg $\bcW$.

\subsubsection*{Organization of the paper}
The remainder of the paper is organized as follows.  In
\Cref{section:admissible-convergence-GMRES}, we review \gmres and the
specification of admissible convergence behavior.  In
\Cref{section:admissible-convergence-rGMRES}, we explain how this has been
extended to allow for specification of admissible convergence of restarted
\gmres. We then explain in \Cref{section:admissible-convergence-blGMRES} how
the theory from \Cref{section:admissible-convergence-GMRES} was extended to the
\blgmres case by interpreting it as an iteration over a right-vector space of
block vectors with scalars from a $^\ast$-algebra.  All of this allows us to
extend existing theory to restarted \blgmres in
\Cref{section:admissible-convergence-rblGMRES}. We end by offering some
conclusions in \Cref{section:conclusions}.

\section{\gmres and admissible convergence}\label{section:admissible-convergence-GMRES}
We first consider the case of $p=1$, i.e., applying a \gmres \cite{SaadSchultz:1986:1} iteration to solve 
\begin{align}
	A\bx 
	=
	\bb
	\in\C^n,
	\label{eqn:Axb}
\end{align}
for which we consider \Wlog the trivial initial approximation $\bx_0=\bnull$.
Given $A$ and $\bb$, the $m$th Krylov subspace is defined
\begin{align}
	\cK_j(A,\bb) 
	= 
	\text{span}\paren[auto]{\{}{\}}{\bb, A\bb,\ldots,A^{j-1}\bb}.
	\label{eqn:Krylov-matrix}
\end{align}
Thus, $\bu \in \cK_j(A,\bb)$ is such that
\begin{align*}
	\bu = p(A)\bb
\end{align*}
where $p(x)$ is a polynomial of $\deg p < j$.
\begin{definition}\label{definition:krylov-matrix}
	The basis $\paren[auto]{\{}{\}}{\bb, A\bb,\ldots,A^{j-1}\bb}$ is called a \textbf{Krylov basis} and
a matrix with Krylov basis as columns is called a \textbf{Krylov matrix},
	\begin{align}
		K_j
		=
		\begin{bmatrix}
			\bb & A\bb & \cdots & A^{j-1}\bb
		\end{bmatrix}
		\in\C^{n\times j}.
		\label{eqn:krylov-matrix}
	\end{align}
\end{definition}

At iteration $j$, the iterate $\bx_j=p(A)\bb\in\cK_j(A,\bb)$ is chosen such that the norm of the residual $\br_j=\bb-A\bx_j$ satisfies the GMRES minimum
residual constraint  
\begin{align} 
	\bx_j 
	= 
	\argmin_{\bx\in \cK_j(A,\bb)}\norm{\bb - A\bx}.  
	\label{eqn:GMRES-LS-constraint-full}
\end{align} 
Minimizing the residual is equivalent to setting $\br_j = \bb-A\bx_j=q_j(A)\bb$,
where the polynomial $q_j(z)$ satisfies 
\begin{align*}
	q_j(z) 
	= 
	\argmin_{q\in\cP_j}
	\norm{q(A)\bb}
\end{align*}
with $\cP_j$ being the set of all polynomials 
of degree at most $j$ satisfying $q(0) = 1$.
We note that \gmres is a projection method, and \eqref{eqn:GMRES-LS-constraint-full} can be equivalently formulated as the
residual orthogonality condition 
\begin{align}
	\mbox{select}
	\ 
	\bx_j
	\in
	\cK_j(A,\bb)
	\ 
	\mbox{such that}
	\ 
	\bb - A\bx	
	\perp
	A\cK_j(A,\bb).
	\label{eqn:GMRES-resid-orth-constraint}
\end{align}
In the classical implementation described in \cite{SaadSchultz:1986:1}, one first builds an orthonormal
basis for $\cK_j(A,\bb)$.
The \emph{Arnoldi process} builds this orthonormal basis 
\begin{align*}
	V_j = \begin{bmatrix} \bv_1, & \bv_2 \ldots , & \bv_j \end{bmatrix}\in\C^{n\times j}.
\end{align*}
This basis satisfies the well-known Arnoldi relation
\begin{align}
	AV_j = V_{j+1}\underline{H_j}=V_jH_j + \alpha\bv_{j+1}\be_j^{T}
	\label{eq:arnoldi-relation}
\end{align}
where $\be_j$ is the $j$th canonical basis vector and
$\underline{H_j}\in\C^{(j+1)\times j}$ and $H_j\in\C^{j\times j}$ are
upper-Hessenberg.  We represent the \gmres approximation as $\bx_j=V_j\by_j$,
where $\by_j\in\C^j$ are its coefficients in the Arnoldi basis. We obtain
$\by_m$ by enforcing the constraint \eqref{eqn:GMRES-LS-constraint-full} which
reduces to solving 
\begin{align}
	\by_j = \argmin_{\by\in\R^m}\norm[auto]{\underline{H_j}\by - \norm[auto]{\br_0}\be_1}_2.
	\label{eq:gmres-ls}
\end{align}

\subsection{Admissible and attainable convergence behavior}
There has been much work done to understand what determines the residual convergence behavior of \gmres, by which we mean the
sequence of residual norms $\curly{\norm[auto]{\br_j}}_{j=0}^{n-1}$ generated by a \gmres iteration.  In particular, many results
have been developed to understand the role eigenvalues can play in determining this sequence and its rate of convergence towards
zero for a non-normal matrix $A$.  
\begin{definition}\label{definition:admissible-convergence}
	A positive, monotonically non-increasing sequence $f_0\geq f_1\geq \cdots \geq f_{n-1} > 0$ is called an
	\textbf{admissible} \gmres convergence sequence. 
\end{definition}
It has been proven that any admissible sequence can be realized by \gmres and
that, pathologically, eigenvalues need not play any role in determining
residual convergence behavior of \gmres \cite{GreenbaumPtakStrakos:1996:1}
\footnote{It should be noted that Meurant showed that while the eigenvalues play no role in the 
residual convergence, they do influence the actual error convergence behavior \cite{Meurant:2012:1}}.
Subsequent work refines these results
\cite{ArioliPtakStrakos:1998:1,DuTebbensMeurant:2017:1,DuintjerTebbensMeurant:2019:1,DuintjerTebbensMeurant:2013:1,MeurantDuintjerTebbens:2014:1,TebbensMeurant:2012:1,MatalonSpillane:2025:1},
and in particular it has also been shown that the Ritz values at each iteration
$j$ (i.e., the eigenvalues of $H_j$) can in pathological cases also be
unrelated to the eigenvalues of $A$ and also to the residual convergence
behavior \cite{TebbensMeurant:2012:1}.

\begin{proposition}[paraphrasing results from \cite{GreenbaumPtakStrakos:1996:1,TebbensMeurant:2012:1}]
	\label{proposition:gmres-convergence-specification}
	For any admissible sequence $\braces{f_j}_{j=0}^n$ defined as in
	\Cref{definition:admissible-convergence}, a matrix/right-hand-side pair
	$(A,\bb)$ can be constructed such that $A$ can be specified to have any
	set of eigenvalues; \gmres applied to \eqref{eqn:Axb} with
	$\bx_0=\bnull$ produces the residual norm sequence $\norm[auto]{\br_j}
	= f_j$. If there is no instance of stagnation (\ie $f_j\neq f_{j+1}$
	for $j=0,1,\ldots n-1$), the Ritz values at each step of the underlying
	Arnoldi process similarly can be arbitrarily specified to be any
	non-zero values. If there is an instance of stagnation at some
	iteration, then at least one Ritz value at that iteration must take the
	value of zero.
\end{proposition}

We assume for simplicity that $A$ is diagonalizable, and we restrict ourselves to the case that \gmres runs for a full $n$ iterations
\footnote{The case of early-terminating \gmres was characterized in \cite{DuintjerTebbensMeurant:2013:1}}. The Arnoldi iteration
constructs a sequence of partial orthogonal Hessenberg factorizations of $A$.  Carried to step $j=n$, it generically produces the full factorization
$A = VHV^\ast$, where $V:=V_n\in\C^{n\times n}$ is unitary, and $H:=H_n\in\C^{n\times n}$ is upper Hessenberg and similar to $A$.  Let
$C\in\C^{n\times n}$ be the companion matrix associated to the characteristic polynomial of $A$, $p_A(z)=\det(zI - A)=z^n -
\sum_{i=0}^{n-1} c_i z^i$. This matrix has the same eigenvalues as $A$ with multiplicity, and $\deg p_A= n$.  This matrix has the
structure
\begin{align*}
	C
	=
	\begin{bmatrix}
		0 & 0 & \dots & 0 & c_0 \\
		1 & 0 & \dots & 0 & c_1 \\
		0 & 1 & \dots & 0 & c_2 \\
		\vdots & \vdots & \ddots & \vdots & \vdots \\
		0 & 0 & \dots & 1 & c_{n-1}
	\end{bmatrix}
\end{align*}
and is also similar to $A$.  Extending the Krylov
matrix \eqref{eqn:krylov-matrix} to the $j=n$ case (denoting $K:=K_n$), we have the Krylov companion matrix relation 
\begin{align*}
	AK
	=
	KC
	\iff 
	A 
	=
	KCK^{-1}.
\end{align*}
We note that the full set of Arnoldi vectors is the unitary factor in a \qr-factorization of $K$; i.e.,
\begin{align*}
	K 
	= 
	V
	(DU)
\end{align*}
with the upper triangular factor being decomposed into a diagonal component $D$
and a unit-diagonal upper triangular component $U$. This allow us to express the
coefficient matrix according to the decomposition
\begin{align}
	A
	=
	V\underbrace{DUCU^{-1}D^{-1}}_{H}V^{\ast}.
	\label{eqn:A-full-gmres-decomp}
\end{align}
 
For every iteration $j$ except $j=n$, there are $j$ Ritz values.  It has been
shown that the strictly upper triangular entries  of each column of
\begin{align*}
	U^{-1}
	=
	\begin{bmatrix}
	    1 & c_0^{(1)} & c_0^{(2)} & c_0^{(3)} & \cdots & c_0^{(n-1)} \\
	    0 & 1 & c_1^{(2)} & c_1^{(3)} & \cdots & c_1^{(n-1)} \\
	    0 & 0 & 1 & c_2^{(3)} & \cdots & c_2^{(n-1)} \\
	    0 & 0 & 0 & 1 & \cdots & \vdots \\
		\vdots & \vdots & \vdots & \vdots & \ddots & c_{n-2}^{(n-1)} \\
	    0 & 0 & 0 & 0 & 0 & 1
	\end{bmatrix}
\end{align*}
are the coefficients of the characteristic polynomials of $H_j$ for each $j$
\cite{TebbensMeurant:2012:1}.  In other words, the characteristic polynomial of
$H_j$ for $1\leq j < n$ is
\begin{align*}
	p_j(z)
	=
	\det(zI - H_j)
	=
	z^j - \sum_{i=0}^{j-1} c_i^{(j)} z^i
	.
\end{align*}
  
Once the Hessenberg matrix $H$ is constructed, the Arnoldi vectors can be specified by performing a unitary transformation with $V$, the matrix of Arnoldi vectors.  However,  it is more useful to consider
fixing the orthonormal basis for the \emph{residual} Krylov subspace. As we only consider the case in which we construct $A$ to be
non-singular, we observe that this is equivalent to specifying $\bV$. Let $\paren[auto]{\{}{\}}{\bw_1,\bw_2,\ldots,\bw_n}$ be orthonormal basis
for $\C^n$, which we take to be the columns of unitary $W\in\C^{n\times n}$ such that  
\begin{align*}
	\mathrm{span}\paren[auto]{\{}{\}}{\bw_1,\bw_2,\ldots,\bw_j} = A\cK_j(A,\bb)
\end{align*}
for all $j\leq n$.

The specification of an admissible \gmres residual convergence is used to
determine an appropriate $\bb$.    If we express the right-hand side with
respect to this basis $\bb=\sum_{i=1}^n h_i\bw_i$, then we observe that the
residual orthogonality constraint \eqref{eqn:GMRES-resid-orth-constraint}
implies that $\br_j=\sum_{i=j+1}^n h_i\bw_i$. Exploiting basic properties of
norms and orthonormal bases, it is observed in
\cite{GreenbaumPtakStrakos:1996:1} that we can thus write $|h_i|^2 =
\norm[auto]{\br_{i-1}}^2 - \norm[auto]{\br_i}^2$, for each $i$.  Thus we can
construct $\bb$ in this basis by determining the $n$ coefficients
$\curly{h_i}$, and we can choose them to obtain the wanted \gmres residual
convergence behavior.  It is shown how to relate the diagonal entries of
$D$ to the specified \gmres convergence pattern,  
\begin{align}
	\be_1^T(DU)^{-1}\be_1 
	=& 
	\be_1^T(DU)^{-\ast}\be_1 
	=
	\be_1^TD^{-1}\be_1 
	=
	\norm[auto]{
		\br_0}^{-1
	}
	\nonumber
	\\
	|\be_i^T(DU)^{-*}\be_1| 
	=&
	\sqrt{\dual[auto]{\br_{i-1}}{\br_{i-1}}^{-1} -\dual[auto]{\br_{i-2}}{\br_{i-2}}^{-1}\,},
	\ 
	i=2,\dots,j.
	\label{eq:first_row_form}
\end{align}

The constructions in \cite{TebbensMeurant:2012:1} show that the Arnoldi basis $V$
and the orthogonal residual Krylov basis $W$ can be related to one another using elements of the matrix $U$ and the prescribed
convergence behavior.  This means that if one specifies $W$ and all other quantities, then the Arnoldi vectors $V$ are uniquely
determined.  Conversely, if along with eigenvalues, Ritz values and a prescribed \gmres residual convergence pattern, the full set
of Arnoldi vectors $V$ is specified, then this uniquely specifies $W$ in such a way that $\bv_1 = \bb/\norm[auto]{\bb}$ and
$\bb=\sum_{i=1}^n h_i\bw_i$.  The precise relationships between the bases $V$ and $W$ are shown in, \eg \cite{KubinovaSoodhalter:2020:1,TebbensMeurant:2012:1}.

\section{Prescribing behavior of restarted \gmres}\label{section:admissible-convergence-rGMRES}
In \cite{DuintjerTebbensMeurant:2019:1}, the authors extend this framework to
generate pairs $(A,\bb)$ that generate a prescribed admissible restarted \gmres
behavior. By restarted \gmres, we mean after running $m$ iterations of \gmres,
the process is truncated, all Arnoldi vectors are discarded, and the residual
$\br_m$ is used as the initial residual for a new cycle of \gmres. This work
has been extended in \cite{MeurantDuintjerTebbens:2020:1} and now also treats
the case of non-diagonalizable matrices. We also note that general theory and
construction of Hessenberg matrices has been treated in the recent book
\cite{Meurant:2025:1}. As this is framed in terms of a matrix decomposition, the authors of
\cite{DuintjerTebbensMeurant:2019:1} also assume that the restarted iteration
terminates within $n$ iterations. In this setting, the authors determine some
restrictions on admissible convergence sequences for restarted \gmres.  As we
are extending the results from \cite{DuintjerTebbensMeurant:2019:1} to the case
of restarted \blgmres, we give a summary of relevant results.  However, we also
provide some new proofs of these results to facilitate extension to the
restarted \blgmres case.  We offer them also for general mathematical interest.  

We denote quantities corresponding to the $k$th restart cycle by superscript $\cdot^{(k)}$. To simplify notation, we assume, without loss of generality, that all the cycles have
the same length $m$ and that $n/m = \ell\in\N$. All the results can be generalized to restarted \gmres with variable cycle length in a
straightforward way.

\subsection{Admissible convergence and the mirroring of stagnation}
In \cite[Section 2]{DuintjerTebbensMeurant:2019:1}, the authors observe that
unlike a full \gmres iteration, not every sequence of non-increasing residual
norms is possible for restarted \gmres.  Instead, they prove that if a restart
cycle ends with $s$ iterations of residual norm stagnation, the next
cycle must begin with $s$ additional iterations of stagnation. We present a
simplification of their proof.

\begin{proposition}\label{prop:simplified-stagnation-mirroring}
	Let there be stagnation of length $s$ at the end of the $k$th cycle, i.e.
	\begin{align}\label{eq:mirror-end}
		\norm[auto]{
			\br^{(k)}_{m-s}
		} 
		=
		\norm[auto]{
			\br^{(k)}_{m}
		},
	\end{align}
	then there is stagnation of the same length at the beginning of the $(k+1)$st cycle, i.e.,
	\begin{align*}
		\norm[auto]{
			\br^{(k+1)}_{0}
		} 
		=
		\norm[auto]{
			\br^{(k+1)}_{s}
		}. 
	\end{align*}
\end{proposition}
The simplified proof observes that the mirroring of stagnation is a consequence of the polynomial minimization characterization of
\gmres, and this makes the result easier to comprehend. Furthermore, this proof is easier to extend 
to the restarted \blgmres setting.
\begin{proof}[Proof of \Cref{prop:simplified-stagnation-mirroring}]
	Let us assume \eqref{eq:mirror-end} holds. When considering stagnation
	in terms of the \gmres polynomial minimization problem, it follows that
	\begin{align*}
		\min_{p\in\cP_m}\norm[auto]{p(A)\br^{(k)}_{0}} 
		=
		\min_{p\in\cP_{m-s}}\norm[auto]{p(A)\br^{(k)}_{0}}.
	\end{align*}
	Denoting the specific minimization polynomial from the stagnating iteration 
	\begin{align*}
		\hat{p}_{m-s} 
		= 
		\argmin_{p\in\cP_{m-s}}\norm[auto]{p(A)\br^{(k)}_{0}},
	\end{align*}
	we obtain an expression for the residual 
	\begin{align*}
		\br^{(k+1)}_{0} 
		=
		\br^{(k)}_{m}
		=
		\hat{p}_{m-s}(A)r^{(k)}_{0}.
	\end{align*}
	We conclude the proof by studying the necessary structure of the minimizing
	polynomial after $s$ iterations in cycle $k+1$,
	\begin{align*}
		\norm[auto]{
			\br_{s}^{(k+1)}
		} 
		=
		\min_{p\in\cP_s}\norm[auto]{p(A)\br^{(k+1)}_{0}}
		=&
		\min_{p\in\cP_s}\norm[auto]{p(A)\hat{p}_{m-s}(A)\br^{(k)}_{0}}
		\\
		\geq&
		\min_{p\in\cP_m}\norm[auto]{p(A)\br^{(k)}_{0}} 
		=
		\norm[auto]{
			\br^{(k)}_{m}
		}.
	\end{align*}
	Due to the minimum residual property of \gmres, it must follow from the inequality that $\norm[auto]{\br^{(k+1)}_{s} } =
	\norm[auto]{\br^{(k)}_{m}}$.
\end{proof}
\begin{remark}
	One implication of this result is that if $s$ iterations of stagnation occur at the end of a \gmres cycle, subsequent
	cycles of length $\leq s$ provide no improvement of the residual.
\end{remark}

\subsection{Building a restarted \gmres factorization}\label{section:build-restarted-kryl-factorization}
We follow the approach taken in \cite{DuintjerTebbensMeurant:2019:1}, which is to build a factorization of $A$ using a restarted
version of the Krylov matrix from \Cref{definition:krylov-matrix} and show that arbitrary admissible behavior of restarted \gmres can be specified
independent of eigenvalues of $A$ and of all the sets of Ritz values generated during this iteration.  From the notation we use
for each cycle of restarted \gmres, it is clear that each cycle induces its own Krylov matrix, i.e., restart cycle $j$ generates
Krylov matrix $K^{(j)}$ defined according to
\begin{align*}
	K^{(1)}
	:=& 
	\begin{bmatrix}
		\bb & A\bb & \cdots & A^{m-1}\bb
	\end{bmatrix}
	\\
	K^{(2)}
	:=& 
	\begin{bmatrix}
		\br_m^{(1)} & A\br_m^{(1)} & \cdots & A^{m-1}\br_m^{(1)}
	\end{bmatrix}
	\\
	\vdots&\\
	K^{(\ell)}
	:=& 
	\begin{bmatrix}
		\br_m^{(\ell-1)} & A\br_m^{(\ell-1)} & \cdots & A^{m-1}\br_m^{(\ell-1)}
	\end{bmatrix}.
\end{align*}
\begin{definition}\label{definition:restarted-krylov-matrix}
	The columns of $K^{(1)}$, $K^{(2)}$, \ldots, $K^{(\ell)}$ taken together span some subspace of $\C^n$. 
	\begin{itemize}
		\item 
			If they are linearly independent, we call this
			collection of vectors a \textbf{restarted Krylov basis}
			of $\C^n$.
		\item 
			The matrix that has these vectors as its (possibly
			linearly dependent) columns
			\begin{align}
				\widetilde{K} 
				:=&
				\begin{bmatrix}
					K^{(1)} & K^{(2)} & \cdots & K^{(\ell)}
				\end{bmatrix} 
				\in\C^{n\times n}.
				\label{eq:rest-kryl-mat}
			\end{align}
			is called a \textbf{restarted Krylov matrix}.
	\end{itemize}
\end{definition}
We note that the analysis of restarted Krylov subspace convergence in the
context of understanding the range of the restarted Krylov matrix as compared
to the Krylov matrix has been considered in, \eg, \cite{DuintjerTebbensMeurant:2019:1,Simoncini:2000:1}.
\begin{proposition}\label{proposition:krylov-matrix-rank}
	Suppose that the Krylov matrix satisfies the assumption 
	\begin{align*}
		\rank\,({K}) 
		=
		n.
	\end{align*}	
	Then it follows that the restarted Krylov matrix \eqref{eq:rest-kryl-mat} is full rank, \ie
	\begin{align*}
		\rank\,(\widetilde{K}) 
		=
		n,
	\end{align*}
	if and only if there is no stagnation at the end of any of the restart cycles.
\end{proposition}
\begin{proof}
	From the definition of the restarted Krylov matrix, we have
	\begin{align*}
		\range(\widetilde{K})\subseteq\range({K}).
	\end{align*}
	Without loss of generality, we consider one iteration of stagnation at the end of the cycle.
	If there is stagnation at the end of the $k$th cycle, then
	\begin{align*}
		\br_m^{(k)} = \br_{m-1}^{(k)}: 
		=
		\hat{p}_{m-1}(A)\br_0^{(k)} \in\cK_{m}(A,\br_0^{(k)}),
	\end{align*}
	and this leads to no increase in rank, i.e.,  
	\begin{align*}
		\rank\paren[auto]{(}{)}{
			\begin{bmatrix}
				K^{(k)} & \br_m^{(k)}
			\end{bmatrix}
		} 
		\leq 
		m. 
	\end{align*} 
	
	Conversely, if there is no stagnation at the end of any restart cycle,
	we proceed by induction. It follows that $K^{(1)}$ has full rank since
	$K$ has full rank. Furthermore, since there is no stagnation at the end
	of the first cycle, the residuals satisfy 
	\begin{align*}
		\Span\braces{\br_0^{(1)}, \br_1^{(1)}, \ldots, \br_m^{(1)}} 
		= 
		\cK_{m+1}(A,\bb), 
	\end{align*}
	and thus $\br_m^{(1)}\not\in \cK_{m}(A,\bb)$.
	Therefore, $A^j\br_m^{(1)} \not\in \cK_{m+j}(A,\bb)$, $j =
	0,\ldots,m-1$. This gives
	\begin{align*}
		\rank\paren[auto]{(}{)}{
			\begin{bmatrix}
				K^{(1)} & K^{(2)}
			\end{bmatrix}
		} 
		=
		2m,
	\end{align*}
	and we proceed analogously for the subsequent cycles.
\end{proof}

\subsection{Residual components in the Arnoldi vectors}
For a \gmres iteration, we express the residual $\br_m$ in terms of Arnoldi vectors computed by the
restarted Arnoldi algorithm. Recall  that the upper Hessenberg matrix $H$ generated by the
Arnoldi algorithm run to completion can be factorized as shown in \eqref{eqn:A-full-gmres-decomp}.  It was shown in \cite[Lemma
2]{DuintjerTebbensMeurant:2019:1} that for all $j$, principal submatrices $H_j$ of $H$ also exhibit such a factorization,
\begin{align}
	H_j
	=
	D_jU_jC_j\left(D_jU_j\right)^{-1},
	\label{eq:Hess_fact}
\end{align}
where $U_j^{-1}$ and $D_j$ are the $j\times j$ principal submatrices of $U^{-1}$ and $D$, respectively. 
We present a lemma (\cite[Lemma 2]{DuintjerTebbensMeurant:2019:1}) showing how the \gmres residual
can be expressed in terms of these quantities, along with a simplified proof. 
\begin{lemma}
	Let $D_m$ and $U_m$ be the factors from \eqref{eq:Hess_fact} for $H_m$ at iteration $m$. Then it follows that
	\begin{align}\label{eq:residual-coeff}
		\br_m 
		= 
		V_{m+1}\bg_m, \quad \text{where} \quad \bg_m :
		=
		\brak{\be_1^T(D_mU_m)^{-1}}^*\dual[auto]{\br_m}{\br_m},
	\end{align}
	where $\cdot^\ast$ denotes taking the conjugate transpose.
\end{lemma}
\begin{proof}
	This proof simplifies the proof of \cite[Lemma 2]{DuintjerTebbensMeurant:2019:1}. 
	The coefficient vector $\bg_m$ in
	\eqref{eq:residual-coeff} satisfies
	\begin{align}
		\dual[auto]{\bg_m}{\bg_m} 
		=
		\dual[auto]{\br_m}{\br_m}.
		\label{eq:ip-coeff-verif}
	\end{align}
	It follows from \eqref{eq:ip-coeff-verif} and \eqref{eq:first_row_form} that we can reduce 
	\begin{align*}
		\dual[auto]{\bg_m}{\bg_m}
		&=
		\norm[auto]{\br_m}^2
		\be_1^T
		\prn{
			D_m U_m
		}^{-1}
		\prn{
			D_m U_m
		}^{-\ast}
		\be_1
		\norm[auto]{\br_m}^2
		\\
		=&
		\norm[auto]{\br_m}^2
		\prn{
			\norm[auto]{\br_0}^{-2}
			+
			\sum_{i=1}^{m-1}
			\prn{
				\norm[auto]{\br_i}^{-2}
				-
				\norm[auto]{\br_{i-1}}^{-2}
			}
		}
		\norm[auto]{\br_m}^2.
	\end{align*}
	To prove $\br_m = V_{m+1}\bg_m$, it suffices to show that
	$\dual[auto]{\br_m}{V_{m+1}\bg_m} = \dual[auto]{\br_m}{\br_m}$, as this
	would imply that $\cos\prn{\br_m,V_{m+1}\bg_m}=1$. Since the minimizer
	(and therefore $\bg_m$) is unique and \eqref{eq:ip-coeff-verif} holds,
	we would conclude it must be the case that the left equation in
	\eqref{eq:residual-coeff} holds. Proving the right-hand equation in \eqref{eq:residual-coeff} 
	then would follow from \eqref{eq:first_row_form}.

	Let $\underline{H_m}$ be the Hessenberg matrix defined as
	in \eqref{eq:arnoldi-relation}. We observe that $\underline{H_m} = H_{m+1}\underline{I_m}$ where $\underline{I_m}$ is the
	$m\times m$ identity matrix with a row of zeros appended at the bottom.  From \eqref{eq:Hess_fact}, it follows that 
	\begin{align}
		\underline{H_m} = D_{m+1}U_{m+1}C_{m+1}(D_{m+1}U_{m+1})^{-1}\underline{I_m}.
		\label{eq:Hm-rect-decomp}
	\end{align}
	From \eqref{eq:gmres-ls}, we can expand the residual $\br_m$ with respect to the basis $V_{m+1}$ where the coefficients
	can be written in terms of the pseudoinverse of $\underline{H_m}$,
	\begin{align*}
		\br_m
		=&
		\bb 
		- 
		AV_m\by_m 
		\\
		=&
		V_{m+1}\prn{
			\norm[auto]{
				\br_0
			}
			\be_1
		}
		- 
		V_{m+1}
		\underline{H_m}\underline{H_m}^\dagger
		\prn{
			\norm[auto]{
				\br_0
			}
			\be_1
		}
		\\
		=&
		V_{m+1}\left(I_{m+1} - \underline{H_m}\,\underline{H_m}^\dagger\right)\be_1\norm[auto]{\br_0}.
	\end{align*}
	We can express the projector $\underline{H_m}\,\underline{H_m}^\dagger$ using \eqref{eq:Hm-rect-decomp} and collapse the
	inner terms to simplify thusly,
	\begin{align*}
		\underline{H_m}\,\underline{H_m}^\dagger
		=&
		D_{m+1}U_{m+1}C_{m+1}
		\begin{bmatrix}
			I_m & U_m^{-1}\bu_{m+1}
			\\
			\bnull^T & 0
		\end{bmatrix}
		C_{m+1}^{-1}(D_{m+1}U_{m+1})^{-1},
	\end{align*}
	whereby $\bnull\in\C^m$ is a vector of zeros. We write in compressed form 
	\begin{align*}
		U_{m+1}
		=
		\begin{bmatrix}
			U_m & \bu_{m+1}
			\\ 
			\bnull^T & 1
		\end{bmatrix}
	\end{align*}
	If we additionally express the companion matrix in compressed form, we can obtain a compressed form for its inverse, i.e.,
	\begin{align*}
		C_{m+1}
		=
		\begin{bmatrix}
			\bnull^T & c_0^{(m+1)}
			\\
			I_m & \bc_{1:m}^{(m+1)}
		\end{bmatrix}
		\implies
		C_{m+1}^{-1}
		=
		\begin{bmatrix}
			-\bc_{1:m}^{(m+1)}/c_0^{(m+1)} & I_m
			\\
			1/c_0^{(m+1)} & \bnull^T
		\end{bmatrix}
		,
	\end{align*}
	where $\bc_{1:m}^{(m+1)}\in\C^m$ is the vector containing the
	coefficients of the characteristic polynomial of for degrees $1$ to
	$m$, with the constant coefficient being $c_0^{(m+1)}$. 
	\footnote{
			We remind the reader that the characteristic polynomial
			is monic with an assumed leading coefficient of $1$
			that is not represented in the companion matrix.
	}
	Multiplying out the middle terms yields the simplification
	\begin{align*}
		C_{m+1}
		\begin{bmatrix}
			I_m & U_m^{-1}\bu_{m+1}
			\\
			\bnull^T & 0
		\end{bmatrix}
		C_{m+1}^{-1}
		=
		\begin{bmatrix}
			0 & \bnull^T
			\\
			(\bc^{(m)} - \bc_{1:m}^{(m+1)}) & I_m
		\end{bmatrix}
		,
	\end{align*}
	where $\bc^{(m)} = U_m^{-1}\bu_{m+1}$.
	From this we finally obtain
	\begin{align*}
		\underline{H_m}\,\underline{H_m}^\dagger
		=
		D_{m+1}U_{m+1}
		\begin{bmatrix}
			0 & \bnull^T
			\\
			(\bc^{(m)} - \bc_{1:m}^{(m+1)}) & I_m
		\end{bmatrix}
		U_{m+1}^{-1}D_{m+1}^{-1}.
	\end{align*}
	Inserting this into the inner product \eqref{eq:ip-coeff-verif} and simplifying
	\begin{align*}
		\dual[auto]{\br_m}{V_{m+1}\bg_m} 
		=& 
		\dual[auto]{\br_m}{\br_m}\left(\be_1^T(DU)^{-1}\be_1\norm[auto]{\br_0} -0\right) 
		\\
		=& 
		\dual[auto]{\br_m}{\br_m}\left(1 - 0\right) 
		=
		 \dual[auto]{\br_m}{\br_m}
	\end{align*}
	completes the proof.
\end{proof}
From \cref{eq:residual-coeff}, we further observe that it follows
\begin{align*}
	\begin{bmatrix}
		V_m&\dfrac{\br_m}{\norm[auto]{\br_m}}
	\end{bmatrix} 
	=
	V_{m+1}
	\begin{bmatrix}
		\begin{matrix}
			I_m
			\\
			\bnull^T
		\end{matrix}
		&
		\dfrac{
			\bg_m
		}
		{
			\norm[auto]
			{
				\bg_m
			}
		} 
	\end{bmatrix}.
\end{align*}

\subsection{Generating the prescribed Hessenberg matrices in restarted Arnoldi}\label{section:prescribing-restarted-gmres}
To simplify the presentation, we characterize the case in which there is no stagnation at the end of any
cycle.  This provides us theory which can be modified for the case of end-of-cycle stagnation,

\subsubsection{No stagnation at the end of any cycle}
We show that if there is no stagnation at the end of a cycle, we can prescribe arbitrary convergence curve for
the next cycle.
\begin{theorem}\label{theorem:restart-gmres-nostag-assignment}
	Let $\left\{\left\{f_j^{(k)}\right\}_{j=0}^{m-1}\right\}_{k=1}^\ell$ be a sequence of positive real numbers satisfying
	\begin{align*}
		f_0^{(1)}
		\geq
		f_1^{(1)}
		\geq
		\cdots 
		\geq
		f_{m-1}^{(1)}>f_{0}^{(2)}\geq\cdots 
		\geq
		f_{m-1}^{(\ell-1)}> f_0^{(\ell)}
		\geq
		\cdots 
		\geq
		f_{m-1}^{(\ell)}>0,
	\end{align*}
	\ie it is monotonically non-increasing with strictly decreasing transitions corresponding to restart cycles
	($f_{m-1}^{(j)}>f_{0}^{(j+1)}>0$ for all $j$). Then there exist a matrix $A$ and a right-hand side $\bb$ such that the restarted
	\gmres with cycle length $m$ applied to $(A,\bb)$ yields residuals satisfying
	\begin{align*}
		\norm[auto]{
			\br_j^{(k)}
		} 
		= 
		f_j^{(k)}, 
		\quad 
		j = 0,\ldots, m-1, 
		\quad 
		k
		=
		1,\ldots,l.
	\end{align*}
\end{theorem}
\begin{proof}
	Without loss of generality, we construct an appropriate Hessenberg matrix and starting vector parallel to the first
	standard basis vector $\be_1$ that exhibits the described restarted \gmres behavior.  At the end of the proof, we discuss
	how the construction may be modified to build non-Hessenberg matrices with arbitrary Arnoldi vectors at each cycle.

	Since we have enforced that for each $j$ it holds $f_{0}^{(j+1)}$ satisfies $f_{m-1}^{(j)}>f_{0}^{(j+1)}>0$, we may
	for each cycle define
	\begin{align*}
		H^{(k)} 
		= 
		D^{(k)}U^{(k)}C^{(k)}\left(D^{(k)}U^{(k)}\right)^{-1}\in\C^{(m+1)\times (m+1)}, \quad k
		=
		1,\ldots,l
	\end{align*}
	to be any upper Hessenberg matrices such that \gmres applied to $(H^{(k)},\be_1f^{(k)}_0)$ yields the convergence curve
	$f_0^{(k)}\geq\cdots\geq f_{m-1}^{(k)}>f_{0}^{(k+1)}$. Denote by $\underline{H^{(k)}} =
	H^{(k)}\underline{I_m}\in\C^{(m+1)\times m}$ to be this matrix but with the last column omitted. The means to construct
	such matrices has been well established in, \eg, \cite{ArioliPtakStrakos:1998:1,TebbensMeurant:2012:1}. 
	
	In the context of this construction, we conclude immediately that the Arnoldi vectors at each cycle are subsets of the
	(sometimes scaled) standard basis vectors, \ie
	\begin{align*}
		V^{(1)} 
		:=& 
		\begin{bmatrix}
			\be_{1}
			&
			\cdots
			&
			\be_{m}
		\end{bmatrix}
		\\
		V^{(k+1)} 
		:=& 
		\begin{bmatrix}
			V^{(k)}
			&
			\be_{km+1}
			&
			\cdots
			&
			\be_{(k+1)m}
		\end{bmatrix}
		\begin{bmatrix}
			\frac{\bg^{(k)}}{\norm[auto]{\bg^{(k)}}}
			&
			0
			\\
			&
			I_{m-1}
		\end{bmatrix}, 
		\quad 
		k 
		= 
		0,\ldots,l-1,
	\end{align*}
	where $\bg^{(k)} = \brak{\be_1^T(D^{(k)}U^{(k)})^{-1}}^*\dual[auto]{f_{0}^{(k+1)}}{f_{0}^{(k+1)}}$, as in
	\eqref{eq:residual-coeff}. We observe that matrices $V^{(k)}$ have orthonormal columns for $k=1,\ldots,\ell$. 	
	Furthermore, since $\be_{m+1}^T\bg^{(k)}$ is non-zero for all $k$, the matrix
	\begin{align*}
		\widetilde{V}
		=
		\begin{bmatrix}
			V^{(1)}&V^{(2)}&\cdots&V^{(\ell)}
		\end{bmatrix}
	\end{align*}
	has full rank.
	
	We construct the right-hand side $\bb$ as
	\begin{align*}
		\bb 
		=
		\be_1f^{(1)}_0.
	\end{align*} 
	The matrix $A$ can be constructed via its action on $n$ linearly independent vectors as
	\begin{align}
		AV^{(k)} 
		=& 
		\begin{bmatrix}
			V^{(k)}&\be_{km+1}
		\end{bmatrix}
		\underline{H^{(k)}}
		\quad \mbox{for} \quad 
		k=1,\ldots,l-1
		\nonumber
		\\
		AV^{(\ell)} 
		=& 
		\begin{bmatrix}
			V^{(\ell)}
			&
			\bnull
		\end{bmatrix}
		\underline{H^{(\ell)}}
		=
		V^{(\ell)}
		\begin{bmatrix}
			I_m
			&
			\bnull
		\end{bmatrix}
		\underline{H^{(\ell)}}
		=
		V^{(\ell)}H_{m\times m}^{(\ell)}
		\label{eq:A_def}
	\end{align}
	where $H^{(\ell)}_{m\times m}\in\C^{m\times m}$ contains the first $m$ rows of $\underline{H^{(\ell)}}$.
	From the construction of $H^{(k)}$, $A$ applied to $V^{(k)}\be_1f^{(k)}_0$ yields the convergence curve specified by the
	sequence $\braces{f_j^{(k)}}_{j=1}^m$. Using \eqref{eq:residual-coeff}, we see that 
	\begin{align*}
		V^{(k+1)}\be_1 
		=
		\begin{bmatrix}
			V^{(k)}&\be_{km+1}
		\end{bmatrix}
		\bg^{(k)}\norm[auto]{\bg^{(k)}}^{-1} =\br_m^{(k)}\norm[auto]{\br_m^{(k)}}^{-1} ,
	\end{align*}
	which ties together the individual cycles.
\end{proof}
We note that we can rewrite \eqref{eq:A_def} as
\begin{align}
	AV^{(k)} 
	&= 
	\begin{bmatrix}
		V^{(k)}&\be_{km+1}
	\end{bmatrix}
	\underline{H^{(k)}} &
	=
	\begin{bmatrix}
		V^{(k)}&V^{(k+1)}\be_1
	\end{bmatrix}
	\underbrace{
		\begin{bmatrix}
			\begin{matrix}I_m\\0\end{matrix}&\bg^{(k)}\norm[auto]{\bg^{(k)}}^{-1}
		\end{bmatrix}^{-1}
		\underline{H^{(k)}}
	}_{=:\underline{\widetilde{H}^{(k)}}},
\end{align}
allowing us to define $A$ fully by its action on all restarted Arnoldi vectors
\footnotesize
\begin{align}
	A
	\widetilde{V} 
	= 
	\widetilde{V}
	\underbrace{
		\begin{bmatrix}
			\underline{\widetilde{H}^{(1)}}
			\\
			&
			\underline{\widetilde{H}^{(2)}}
			\\
			&
			&
			\ddots
			\\
			&
			&
			&
			\underline{\widetilde{H}^{(\ell-1)}}
			\\
			&
			&
			&
			&
			H^{(\ell)}_{m\times m}
		\end{bmatrix}
	}_
	{
		=: 
		\widetilde{H}
	}
	.
	\label{eq:big_A}
\end{align}
\normalsize
From this construction, we see that the eigenvalues of $A$ are determined by the Hessenberg blocks in \eqref{eq:big_A}.
\begin{corollary}
	The eigenvalues of $A$ can be recovered as the cumulative eigenvalues of the upper-triangular blocks of 
	\begin{align}
		\begin{bmatrix}
			I_{m}
			&
			0
		\end{bmatrix}
		\underline{\widetilde{H}^{(k)}}
		&
		\quad \mbox{for} \quad  
		k=1,\ldots,\ell-1;
		\nonumber
		\\
		H^{(\ell)}_{m\times m}
		&
		\quad \mbox{for} \quad  
		k=\ell.
		\label{eq:Hessenberg-eig-blocks}
	\end{align}
	\label{corollary:eigs_of_A}
\end{corollary}
\begin{proof}
	One observes that the Hessenberg matrix on the right-hand side of \eqref{eq:big_A} is block lower triangular; \ie, the
	eigenvalues of $A$ are the eigenvalues of their diagonal blocks, which are precisely the blocks defined in \eqref{eq:Hessenberg-eig-blocks}. 
\end{proof}
The Ritz values for each cycle can be similarly determined.
\begin{corollary}
	If a cycle of restarted \gmres is specified to have no stagnating
	iterations, the Ritz values of the individual Hessenberg matrices
	$H^{(k)}$ can be arbitrarily assigned to be any non-zero values.
	\label{corollary:ritzvals_of_A}
\end{corollary}
\begin{proof}
From the construction in the proof of
\Cref{theorem:restart-gmres-nostag-assignment}, one sees that the Ritz values
for each cycle are determined by the strictly upper triangular entries of
$\prn{U^{(k)}}^{-1}$.  The condition of them being non-zero when there is no
stagnation specified follows directly from
\Cref{theorem:restart-gmres-nostag-assignment}.
\end{proof}

\Cref{theorem:restart-gmres-nostag-assignment} and the proof thereof consider
the special case in which we use Hessenberg matrices and their actions on
standard basis vectors to construct a matrix and right-hand side producing the
specified restarted \gmres residual convergence behavior with the specified
eigenvalues as well as Ritz values for each cycle.  This is the strategy used
by other authors (\eg, \cite{GreenbaumPtakStrakos:1996:1}), as it has been observed
that \gmres behavior is unitarily invariant.  This holds in our setting for
each cycle of restarted \gmres for the constructed matrix/right-hand-side pair.
Thus we can modify the construction to generate any $\ell$ collections of
orthonormal restarted Arnoldi vectors, via a basis transformation. 

\begin{corollary} The results from \Cref{theorem:restart-gmres-nostag-assignment} are valid for any basis of mutually independent sets of vectors, whereby for
	each $k$, $V^{(k)}$ has orthonormal columns, and their first columns satisfy 
\begin{align*} 
	V^{(k+1)}\be_1 
	=
	\begin{bmatrix}
		V^{(k)}&\tilde{v}_k
	\end{bmatrix}
	\bg^{(k)}\norm[auto]{\bg^{(k)}}^{-1},
\end{align*}
where $\tilde{v}_k$ is any normalized vector orthogonal to $V^{(k)}$.
\label{rem:gen_arnoldi_vec}
\end{corollary}

\begin{remark}
	We observe that this construction allows one to assign eigenvalues and
	Ritz values for each Arnoldi cycle while simultaneously specifying an
	admissible restarted \gmres residual convergence behavior.  By the its
	very nature, this procedure constructs a matrix and right-hand side; so
	it is restricted to assigning residual norms for the first $n$
	residuals, as observed in \cite{DuintjerTebbensMeurant:2019:1}.  We
	note that this construction does not allow for the specification of the
	last residual norm of the last ($\ell$-th) cycle, \ie
	$\norm[auto]{\br_m^{(\ell)}}$.  We offer some discussion on how one
	might adjust the construction to influence this final residual and,
	through this, note that the construction as presented contains an
	implicit choice, which has been made to allow for unambiguous
	assignment of the eigenvalues and Ritz values.
\end{remark}

By defining $A$ according to its action on the restarted Krylov basis vectors
(as in \eqref{eq:A_def}), there is an implicit choice one must make to close
the definition.  Namely, we must choose how $A$ acts on the last Arnoldi vector
in the last restarted \gmres cycle.  In \cite{DuintjerTebbensMeurant:2019:1},
the authors make the choice shown in the second line of \eqref{eq:A_def}.  We
demonstrate that this choice enables the unambiguous assignment of eigenvalues
$A$ and Ritz values for each restart cycle; \ie, how $A$ acts on the last
column of $V^{(\ell)}$ must be constrained appropriately, for the assignment to
be possible.  Without this constraint, we can influence
$\norm[auto]{\br_m^{(\ell)}}$ at the expense of being able to precisely specify
the eigenvalues and Ritz values.  We elaborate thusly.

Consider that unlike full \gmres, restarted \gmres is not guaranteed to converge at the final iteration.  Therefore, a more
general Arnoldi relation in the final cycle of restarted \gmres would be
\begin{align}
	AV^{(\ell)}
	=
	\begin{bmatrix}
		V^{(\ell)}
		&
		\bu
	\end{bmatrix}
	\underline{H^{(\ell)}}.
	\label{eq:final-arnoldi-next-vector}
\end{align}
For all cycles but the last, we adjust the decomposition to accommodate a
restart using the last residual from the previous cycle as the first Arnoldi
vector of the new cycle. However, at the last iteration of the last cycle, we
have already generated a full, linearly independent basis for $\C^n$, assuming
no end-of-cycle stagnation.  Thus, the vector $\bu$ can be completely
represented in the restarted Krylov basis as
\begin{align}
	\bu
	=
	\widetilde{V}
	\widehat{\bc}
	\label{eq:next-vector-linear-combo}
\end{align}
where $\widehat{\bc}\in\C^n$ is an arbitrary vector.  

We demonstrate that $\widehat{\bc}$ can be chosen freely such that it
indirectly influences the value of $\norm{\br_m^{(\ell)}}$ at the expense of no
longer being able to specify eigenvalues and Ritz values.  We note that in
\eqref{eq:A_def}, the choice is taken to be $\bu = \bnull$, and this choice
allows for them to be specified unambiguously. We can \Wlog assume the case of
no end-of-cycle stagnation. 
\begin{lemma}\label{lemma:rank-one-update-Hessenberg-laststep}
	Let $\bu = \widetilde{V}\widehat{\bc} \neq \bnull$.  This corresponds to a rank-one update of the relation \eqref{eq:A_def}, namely
	\begin{align*}
		A\widetilde{V}
		=
		\widetilde{V}
		\prn{
			\underbrace{
				\widetilde{H}
				+
				\bc\be_n^T
			}_
			{
				=:
				\widetilde{H}(\bc)
			}
		}
	\end{align*}
	where $\bc = h_{m+1,m}^{(\ell)}\widehat{\bc}$, and $h_{m+1,m}^{(\ell)}$ is the $(m+1,m)$th entry of
	$\underline{H^{(\ell)}}$.
\end{lemma}
\begin{proof}
	Each cycle of restarted Arnoldi can be determined independently of the
	others.  This holds also for the final cycle, with the restriction that
	the $(m+1)$st Arnoldi vector of the final cycle is representable as a
	linear combination of the $n=m\ell$ linearly independent restarted
	Arnoldi vectors from the previous $\ell$ cycles.  Defining this vector
	$\bu$ as in \eqref{eq:next-vector-linear-combo}, we insert this
	representation into \eqref{eq:final-arnoldi-next-vector},
	\begin{align*}
		AV^{(\ell)}
		=&
		\begin{bmatrix}
			V^{(\ell)}
			&
			\bu
		\end{bmatrix}
		\underline{H^{(\ell)}}
		\\
		=&
		\widetilde{V}
		\begin{bmatrix}
			\begin{matrix}
				\bnull_{(n-m)\times m}
				\\
				I_m
			\end{matrix}
			&
			\widehat{\bc}
		\end{bmatrix}
		\underline{H^{(\ell)}}
		\\
		=&
		\widetilde{V}
		\prn{
			\begin{bmatrix}
				\bnull_{(n-m)\times m}
				&
				\bnull
				\\
				I_m
				&
				\bnull
			\end{bmatrix}
			+
			\widehat{\bc}\be_{m+1}^T
		}
		\underline{H^{(\ell)}}
		\\
		=&
		\widetilde{V}
		\prn{
			\begin{bmatrix}
				\bnull_{(n-m)\times m}
				\\
				H_{m\times m}^{(\ell)}
			\end{bmatrix}
			+
			\underbrace{
				h_{m+1,m}^{(\ell)}\widehat{\bc}
			}_
			{
				=:\bc
			}
			\be_m^T
		}
		.
	\end{align*}
	This allows for the modification of \eqref{eq:A_def} to adjust the action of $A$ to reflect this more general
	specification of the last Arnoldi step of the final cycle,
	\begin{align*}
		A\widetilde{V}
		=
		\widetilde{V}
		\prn{
			\widetilde{H}
			+
			\bc\be_n^T
		},
	\end{align*}
	completing the proof.
\end{proof}
This allows us to characterize the form of the final residual, which now depends on the choice of $\bc$.
\begin{corollary}\label{corollary:final-residual-expression}
	Let $\bx_m^{(\ell)}=\widetilde{V}\by$, $\by\in\C^n$, be the final restarted \gmres approximation for the $\ell$th cycle.  Then the residual
	$\br_m^{(\ell)} = \bb - A\bx_m^{(\ell)}$ admits the expression
	\begin{align}
		\br_m^{(\ell)}
		=
		\widetilde{V}
		\prn{
			\norm{\bb}\be_1
			-
			\widetilde{H}\by
			-
			y_n\bc
		}
		,
		\label{eqn:residual-shift-structure}
	\end{align}
	where $y_n$ is the $n$th component of $\by$.
\end{corollary}
\begin{proof}
	We compute the residual via the representation of $\bx_m^{(\ell)}$ in the restarted Arnoldi basis,
	\begin{align*}
		\br_m^{(\ell)}
		=&
		\widetilde{V}\be_1
		-
		A\widetilde{V}\by
		\\
		=&
		\widetilde{V}
		\prn{
			\be_1
			-
			\widetilde{H}(\bc)\by
		}
		\\
		=&
		\widetilde{V}
		\prn{
			\be_1
			-
			\widetilde{H}\by
			-
			\bc\be_n^T\by,
		}
	\end{align*}
	which when multiplied out completes the proof.
\end{proof}
\begin{remark}[Influencing rather than specifying $\norm{\br_m^{(\ell)}}$]
	We observe that from \eqref{eqn:residual-shift-structure} it follows
	that tuning $\bc$ will influence $\norm{\br_m^{(\ell)}}$, but this
	influence is not direct.  Since the restarted \gmres basis
	$\widetilde{V}$ does not form an orthonormal basis, it cannot be
	removed from the norm calculation; \ie the Arnoldi vectors in each
	restart cycle are orthonormal to one another, but most vectors from
	different cycles exhibit no orthogonality relation.  In addition, the
	different cycles are coupled as described in
	\Cref{rem:gen_arnoldi_vec}.  Thus, expressing the influence of $\bc$ on
	$\norm{\br_m^{(\ell)}}$ is not straightforward.
\end{remark}
\begin{remark}[The effect of $\bc\neq\bnull$ on eigen/Ritz values]
	The choice of $\bc$ influences only the final restarted \gmres
	approximation. Observe that it does not not effect any other
	approximation or the previous prescribed residual norms. What $\bc\neq
	\bnull$ does effect is all the eigenvalues and Ritz values at every
	cycle.  As $\widetilde{H}(\bc)$ is a rank-one update of
	$\widetilde{H}$, if we prescribe eigenvalues and Ritz values when
	constructing $\widetilde{H}$, we could avail of the theory concerning
	the behavior of eigenvalues with respect to rank-one updates of
	matrices to make some general statements about how close the
	eigenvalues and Ritz values of $\widetilde{H}(\bc)$ to the ones
	prescribed for $\widetilde{H}$, depending on $\norm{\bc}$; see, \eg,
	\cite{HornJohnson:2012:1, StewartSun:1990:1} for some discussion of this topic.  However,
	we cannot represent them in a straightforward manner as functions of
	$\bc$.
	\label{rem:eigs_of_A}
\end{remark}
\begin{remark}
	We note that any further non-trivial iterations of restarted \gmres are
	completely determined by the restarted \gmres factorization we have
	constructed.
\end{remark}

\Cref{theorem:restart-gmres-nostag-assignment} shows existence of a matrix and a right-hand side producing prescribed convergence curve of restarted \gmres.
With \Cref{corollary:eigs_of_A,rem:gen_arnoldi_vec}, we can obtain a complete characterization of these matrices and right-hand sides.

\subsubsection{Stagnation at the end of a cycle}

If there is stagnation of length $s$ at the end of cycle $k$, then the first
$s$ columns of $H^{(k+1)}$ are fully determined by $\underline{H^{(k)}}$. We
also need to define action on further $s$ linearly independent vectors to
uniquely determine $A$ and $\vb$. This corresponds to specifying an equal number of
additional iterations of cycle $\ell+1$ and possibly further cycles depending
on the number of end-of-cycle stagnating iterations.

\section{\blgmres admissible convergence behavior}\label{section:admissible-convergence-blGMRES}
\Cref{section:admissible-convergence-rGMRES} contains a review of the theory
from \cite{DuintjerTebbensMeurant:2019:1}, but we have also presented new
proofs of some of the authors' results. While these alternative proofs are of
general mathematical interest, their primary purpose is to present the theory
of \cite{DuintjerTebbensMeurant:2019:1} in language more amenable to extending
that work to the block \gmres (\blgmres) setting.  

We build heavily on theory developed in \cite{KubinovaSoodhalter:2020:1}
concerning specifying residual convergence behavior of \blgmres.  In
\cite{DuintjerTebbensMeurant:2019:1}, the authors show that specifying the
convergence behavior of restarted \gmres is achieved by specifying the behavior
of individual cycles and assuring that the final residual of one cycle becomes
the first Arnoldi vector (scaled) of the next cycle.  They characterize the
restrictions imposed by end-of-cycle stagnation. We use the theory in
\cite{DuintjerTebbensMeurant:2019:1} directly, with the help of the language
introduced in the alternative proofs in
\Cref{section:admissible-convergence-rGMRES} to make the same characterization
for restarted block \gmres. The main challenge is to properly characterize the
notion of stagnation in the block \gmres setting. 

As in \Cref{section:admissible-convergence-GMRES}, we may assume \Wlog that $\bX_0=\bnull$.
We note that in \cite{KubinovaSoodhalter:2020:1}, the authors do not consider
the case of block Arnoldi breakdown (\ie the case in the block Arnoldi spanning
set ceases to be a basis of $\K_j(A,\bB)$ at some iteration $j$). We adopts the
same assumption in this paper.

We briefly review \blgmres and the theory in \cite{KubinovaSoodhalter:2020:1};
\WLoG we take $n$ to be a multiple of $p$, following
\cite{KubinovaSoodhalter:2020:1}. \blgmres is a generalization of \gmres for
the case when one is solving a system with multiple right-hand sides
\eqref{eqn:AXB}. One generates the block Krylov subspace, which is often
characterized as the block span of $n\times p$ vectors, \ie,
\begin{align}
	\K_j(A,\bB)
	\coloneq &
	\blspan	
	\curly{
		\bB
		,
		A\bB
		,
		A^2\bB
		,
		\ldots
		,
		A^{j-1}\bB	
	}
	\nonumber
	\\
	\coloneq &
	\braces{
		\sum_{i=0}^{j-1}
		A^i \bB S_i
		\quad | \quad 
		\mbox{for all}
		\quad
		\braces{
			S_i
		}_{i=0}^{j-1}
		\subset\C^{p\times p}
	}.
	\label{eqn:block-krylov-def}
\end{align}
\begin{definition}
	We denote as the \emph{block Krylov matrix} associated with the block Krylov subspace \eqref{eqn:block-krylov-def},
	\begin{align*}
		\bcK
		=
		\begin{bmatrix}
			\bB
			&
			A \bB
			&
			A^2 \bB
			&
			\cdots
			&
			A^{\ell-2} \bB
			&
			A^{\ell-1} \bB
		\end{bmatrix}
		.
	\end{align*}
\end{definition}
Let
\begin{math}
	\bcV_j
	=
	\begin{bmatrix}
		\bV_1
		&
		\bV_2
		&
		\cdots
		&
		\bV_j
	\end{bmatrix}
	\in\C^{n\times jp}
\end{math}
and $\bV_\ell\in\C^{n\times p}$, $\ell = 1,2,\ldots, j$, be the orthonormal
basis generated by a block Arnoldi process satisfying the block Arnoldi
relation $A\bcV_j=\bcV_{j+1}\underline{\cH_j}$ where
$\underline{\cH_j}\in\C^{(j+1)p\times jp}$ is block upper-Hessenberg with the
lower subdiagonal blocks being normalizing quantities resulting from the
normalization step of the block Arnoldi procedure. Then for $\bX_j =
\bcV_j\bY_j$, the \blgmres minimizer, $\bY_j\in\C^{jp\times p}$ is
characterized by the Frobenius norm (or equivalently the column-wise $2$-norm)
minimization
\begin{align*}
	\bY_j
	=
	\argmin_
	{\bY\in\C^{jp\times p}}
	\norm[auto]
	{
		\underline{\cH_j}\bY 
		- 
		\bE_1 S_0
	}_F
\end{align*}
where $\bE_1\in\R^{\prn{j+1}p\times p}$ contains the first $p$ columns of the $jp\times
jp$ identity matrix, $\be_i$ is the $i$th standard basis vector, and $\bB = \bV_1 S_0$ is a QR-factorization of the
right-hand side.  A variety of implementations have been developed using these
theoretical underpinnings; see, \eg \cite[Chapter 6.12]{Saad2003Iterative} and
\cite{Gutknecht:2007:1,BakerDennisJessup:2006:1,Vital:1990:1}.  

Considering these methods in terms of block vectors brings advantages for
analysis, by maintaining certain structural properties that are lost when not
fully considering the block structure.  For example, the authors of
\cite{SimonciniGallopoulos:1996:1} show that the block minimization is
connected to a matrix-valued polynomial minimization problem, \cf
\Cref{theorem:block-polynomial-minimization-blgmres}.  
\begin{definition}
	\label{definition:matrix-valued-polynomial-solvents}
	Let $C_k\in\C^{p\times p},\ k=0,1,\ldots, n-1$ be the block coefficients defining
	the \emph{matrix-valued polynomial}
	\begin{align*}
		M(\lambda)
		=
		\lambda^n I
		-
		\sum_{k=0}^{n-1}
		\lambda^k C_k
		\in\C^{p\times p}.
	\end{align*}
	We call $M(\lambda)$ a \emph{$\lambda$-matrix} and for a matrix $X\in\C^{p\times p}$
	\begin{align*}
		M(X)
		=
		X^n
		- 
		\sum_{k=0}^{n-1}
		C_k X^k
	\end{align*}
	a \emph{matrix polynomial}. A number $\tilde{\lambda}$ is called a
	\emph{latent root} if $M\prn{\tilde{\lambda}}$ is singular.  A matrix
	$S\in\C^{p\times p}$ is called a \emph{right solvent} of $M$ if $M(S) =
	0\in\C^{p\times p}$.
\end{definition}
For an $F\in\C^{n\times n}$ we can define the matrix polynomial operator $M(F):\C^{n\times p}\rightarrow\C^{n\times p}$ via
\begin{align*}
	M(F)\circ V
	=
	F^n V
	-
	\sum_{k=0}^{n-1}
	F^k V C_k.
\end{align*}
for $V\in\C^{n\times p}$.
This matrix-valued polynomial is associated to the block companion matrix
\begin{align*}
	\cC
	=
	\begin{bmatrix}
		0 & 0 & \dots & 0 & C_0 \\
		I & 0 & \dots & 0 & C_1 \\
		0 & I & \dots & 0 & C_2 \\
		\vdots & \vdots & \ddots & \vdots & \vdots \\
		0 & 0 & \dots & I & C_{n-1}
	\end{bmatrix},
\end{align*}
whose eigenvalues are the latent roots of the $\lambda$-matrix of the matrix polynomial.

The authors of \cite{KubinovaSoodhalter:2020:1} sought to extend the results of
\cite{GreenbaumPtakStrakos:1996:1,TebbensMeurant:2012:1} to the \blgmres
setting.  They observed that this is more easily accomplished by carefully
working with the block structure, using the tools from
\cite{FrommerLundSzyld:2017:1,SimonciniGallopoulos:1996:1}. The approach
considers \blgmres as an iteration occurring on the right vector space
$\S^N\coloneq\C^{n\times p}$ with elements of $\S\coloneq\C^{p\times p}$ acting
as scalars being multiplied from the right. We can \Wlog assume $n=N p$ so that
we can consider \eqref{eqn:AXB} to be an equation $A\bX=\bB$ over $\S$. We
refer the reader to \cite[Table 1]{KubinovaSoodhalter:2020:1} and the text of
the containing section for a full description of how this right vector space is
structured for the purpose of the analysis.  The salient points are:
\begin{itemize}
	\item upper-triangular matrices with positive (non-negative) diagonal
		entries, which are called \emph{normalizing quantities}, denoted $\S^+$ ($\S^+_0$)
		\cite{FrommerLundSzyld:2017:1}, take the role of positive (non-negative) scalars;
	\item the role of normalization of $\bZ\in\S^N$ is thus taken
		by the economy QR-factorization $\bZ=\bV R$, performed such
		that $R\in\S^+_0$, and we denote $\blnorm{\bZ}=R$;
	\item the inner product generalizes naturally to the block inner
		product, where for $\bV,\bW\in\S^N$, we denote
		$\bldual{\bV}{\bW}\coloneq\bW^\ast\bV\in\S$, observing that this
		means $\bldual{\bV}{\bV}=\blnorm{\bV}^\ast\blnorm{\bV}$; 
	\item the role of absolute value in this setting for $S\in\S$ is $\abs[auto]{S} := \cholu(S^\ast S)$, where \cholu refers
		to the upper-triangular Cholesky factor of
		the argument;
	\item normalizing quantities are partially ordered by connecting them
		to the Loewner ordering of \hpd matrices; \ie, 
	\begin{align*}
		R_1 \leq R_2 \iff R_1^\ast R_1 \preceq R_2^\ast R_2,
	\end{align*}
	in the Loewner sense, for normalizing quantities $R_1,R_2\in\S^+_0$. It
	should be noted that one generally cannot make definitive bounding
	statements or discuss minimization in the context of a Loewner partial
	ordering. However, the mechanics of block GMRES guarantee that the
	particular normalizing quantities associated to block GMRES residual
	norms are ordered.
\end{itemize}
This setup allows for the definition and specification of admissible
convergence behavior of \blgmres in \cite{KubinovaSoodhalter:2020:1}.    
\begin{remark}
	As the authors of \cite{KubinovaSoodhalter:2020:1} do not consider the
	case in which the block Krylov subspace loses dimension, we also omit
	that case and always assume that each block Arnoldi vector has full
	column rank $p$.  This avoids the block version of the restarted Krylov
	basis losing linear independence, except in the case analogous to the
	non-block setting whereby end-of-cycle stagnation leads to linear
	dependence; cf. \Cref{section:bl-stag-mirroring}.
\end{remark}

\begin{theorem}[paraphrasing Theorem 4.1 from \cite{KubinovaSoodhalter:2020:1}]
	Let $\curly{F_0,F_1,\ldots, F_{n-1}}\subset\S^+_0$ be an admissible
	sequence of normalizing quantities; \ie $F_i\geq F_{i+1}$ for all $i$.
	Then a matrix/right-hand-side pair $(A,\bB)$ can be constructed such
	that \blgmres applied to \eqref{eqn:AXB} produces a sequence of
	residuals satisfying $\blnorm{\bR_i}=F_i$, and $A$ can be constructed
	to have any eigenvalues and Ritz values, at each block Arnoldi
	iteration.  The matrix/right-hand-side pair admit the structure 
	\begin{align*}
		A
		=
		\bcV\cD\cU\cC\cU^{-1}\cD^{-1}\bcV^\ast
		\quad\mbox{and}\quad 
		\bB
		=
		\bcV\bE_1 F_0,
	\end{align*}
	where $\bcV$ is a unitary matrix, $\cC$ is the block companion matrix
	corresponding to a matrix-valued polynomial with solvents (a block generalization of eigenvalues; see \Cref{definition:matrix-valued-polynomial-solvents}) encoding the
	eigenvalues of $A$, and $\cU$ is the block upper-triangular matrix with
	identity blocks along the block diagonal encoding the Ritz solvents and therefore the Ritz values,
	and $\cD$ is block diagonal with entries in $\S^+$ encoding the block
	residual norm sequence. The Ritz values can be arbitrarily assigned to
	be any non-zero values, as long as \blgmres is specified to exhibit no
	partial or full stagnation; \cf \Cref{definition:block-stagnation}.
	\label{theorem:block-gmres-specification}
\end{theorem}
For greater detail and full explanation of the theory, see
\cite{KubinovaSoodhalter:2020:1}. For $j=1,2,\ldots, k$, the $jp\times jp$
principal submatrix $\cH_j$ of the block upper Hessenberg $\cH=\cD\cU\cC\cU^{-1}\cD^{-1}$ is
associated to a matrix polynomial with block companion matrix encoded by
$\cU^{-1}$ to define the Ritz solvents (and therefore Ritz values) for each $j$.

\section{Extension to restarted \blgmres}\label{section:admissible-convergence-rblGMRES}
With the theory of \Cref{section:admissible-convergence-GMRES} extending
naturally to the \blgmres case, we explore how this theory extends to the case
of restarted \blgmres.  In \Cref{section:admissible-convergence-rGMRES}, we
have presented and extended the roadmap of \cite{DuintjerTebbensMeurant:2019:1}
for using admissible convergence behavior theory to specify behavior for each
restart cycle and then to connect the cycles together.  In
\cite{DuintjerTebbensMeurant:2019:1}, it is shown that the occurrence of
end-of-cycle stagnating iterations restricts the admissible behavior for a
corresponding number of iterations in the subsequent cycle.  In
\Cref{section:admissible-convergence-rGMRES}, we framed this theory and some of
the proofs using a language that accommodates its extension to the restarted
\blgmres setting.  In the case that there is no stagnating behavior, the
approach discussed in \Cref{section:admissible-convergence-rGMRES} allows us to
extend the theory of \cite{KubinovaSoodhalter:2020:1} without difficulty.  The
issue of stagnation requires development of additional results, in order to
extend the ideas in \Cref{section:admissible-convergence-blGMRES} to the
restarted \blgmres case.  

Let $M$ denote the cycle length for restarted \blgmres.  As with the restarted
\gmres case, \Wlog we can take $M$ to be constant for all cycles.  Furthermore,
let us assume that $N=\ell M$ so that $n=\ell M p$.

\subsection{Mirroring of stagnation behavior}\label{section:bl-stag-mirroring}

For \blgmres, we must clearly define the notion of stagnation.  This was
discussed in \cite{Soodhalter:2017:1}, and we present a more general treatment of
the idea.
\begin{definition}
	For a block Krylov subspace method (\eg \blgmres), if
	$\bX_j=\bX_{j-1}$ then we write that the method has
	\emph{totally stagnated} at iteration $j$.  If
	$\bX_j\neq\bX_{j-1}$ but $\bX_j\bu = \bX_{j-1}\bu$ for some
	$\bu\in\C^p\setminus\braces{\bnull}$, then we say that the method has
	\emph{partially stagnated} with respect to $\bu$.
	\label{definition:block-stagnation}
\end{definition}
As we assume that $\dim\prn{\range{\cK_j}}=jp$ for all $j\leq N$, the occurrence of a
partial stagnation can be clearly characterized in terms of the block Arnoldi
basis.
\begin{lemma}
	Let there be a block stagnation with respect to $\bu$ at iteration $j$.  If we
	express the $j$th \blgmres iterate as $\bX_j=\sum_{i=1}^j\bV_i F_i^{(j)}$,
	then it follows that $F_j^{(j)}$ is singular, and
	$\bu\in\cN\prn{F_j^{(j)}}$.
\end{lemma}
\begin{proof}
	Consider that if we express partial stagnation with respect to
	$\bu$ in terms of the block Arnoldi basis, we obtain
	\begin{align*}
		\sum_{i=1}^j
		\bV_i F_i^{(j)}\bu
		=
		\sum_{i=1}^{j-1}
		\bV_i F_i^{(j-1)}\bu
		\iff
		\sum_{i=1}^{j-1}
		\bV_i
		\prn{
			F_i^{(j-1)} - F_i^{(j)}
		}
		\bu
		=
		\bV_j F_j^{(j)}\bu
	\end{align*}
	As we assume there has been no breakdown of the block Arnoldi process
	and that $\bu\neq\bnull$, it follows from the block Arnoldi vectors
	being an orthonormal basis that both sides of the equation must be zero
	for the equation to hold.  This implies that $F_j^{(j)}\bu=\bnull$.
	\footnote{
		Observe that the orthogonality of the block Arnoldi vectors
		also implies that coefficient differences $F_i^{(j-1)} -
		F_i^{(j)}$ must also be singular with $\bu$ being in the
		intersection of their null spaces
	}
\end{proof}
The conclusion drawn in \cite{Soodhalter:2017:1} is that some subspace of
$\range\prn{\bV_j}$ does not contribute to the \blgmres minimization at
iteration $j$.  In the single-vector case of $p=1$, this reduces down to the
coefficient simply being zero, meaning that the minimization produces no
improvement in the single new Arnoldi direction.  We observe also that it
follows from straightforward algebraic manipulation that a \blgmres
partial stagnation with respect to $\bu$ necessarily means that
\begin{align*}
	\bu^\ast
	\bldual[auto]{\bR_j}{\bR_j}
	\bu
	=
	\bu^\ast
	\bldual[auto]{\bR_{j-1}}{\bR_{j-1}}
	\bu;
\end{align*}
\ie, the block residual normalizing quantity does not completely stagnate, but
it stagnates when tested against $\bu$.

We include \cite[Theorem 3.6]{KubinovaSoodhalter:2020:1} which shows how the
block residual polynomial is obtained as the minimizer in the Loewner sense,
demonstrating why the Loewner-based partial ordering of residual normalizing
quantities shown in \Cref{theorem:block-gmres-specification} is the correct
definition of an admissible convergence sequence. Furthermore, it is useful
when discussing the correct generalization of the stagnation mirroring results
in \Cref{section:admissible-convergence-rGMRES}.


\begin{theorem}{\cite[Theorem 3.6]{KubinovaSoodhalter:2020:1}}
	\label{theorem:block-polynomial-minimization-blgmres}
	Let $\cP_j$ be the set of all $\lambda$-matrices of degree at most $j$.
	The \blgmres residual polynomials $\widehat{P}_j$ satisfy
	\begin{align}
		\widehat{P}_j = 
		\underset{
			p
			\in\cP_j
			\atop
			P(0) 
			=
			I
		}
		{\operatorname{argmin}} \bldual[auto]{P(\bcA )\circ \bB}{P(\bcA )\circ \bB}.
		\label{eq:res_poly_block}
	\end{align}
\end{theorem}

The notion of partial stagnation from \Cref{definition:block-stagnation} can be
developed from these pieces. We present a lemma generalizing the mirroring of
stagnation to the \blgmres case.


\begin{proposition}\label{proposition:block-stagnation-in-a-direction}
	Let there be stagnation of length $s$ at the end of the $k$th cycle with respect to $\bu\in \C^{p}$, i.e.,
	\begin{align*}
		\bu^*\bldual[auto]{\bR^{(k)}_{m-s}}{\bR^{(k)}_{m-s}}\bu 
		=
		\bu^*\bldual[auto]{\bR^{(k)}_{m}}{\bR^{(k)}_{m}}\bu.
	\end{align*}
	Then there is stagnation of the same length at the beginning of the $(k+1)$st cycle also with respect to $\bu$; i.e.,
	\begin{align*}
		\bu^*\bldual[auto]{\bR^{(k+1)}_{0}}{\bR^{(k+1)}_{0}}\bu 
		=
		\bu^*\bldual[auto]{\bR^{(k+1)}_{s}}{\bR^{(k+1)}_{s}}\bu.
	\end{align*}
\end{proposition}
\begin{proof}
	Assume \Wlog for $k=1$, that
	\begin{align*}
		\bu^*\bldual[auto]{\bR^{(1)}_{m-s}}{\bR^{(1)}_{m-s}}\bu 
		=
		\bu^*\bldual[auto]{\bR^{(1)}_{m}}{\bR^{(1)}_{m}}\bu.
	\end{align*}
	Then (denoting $\bR_0^{(1)}\coloneq \bB$)
	\begin{align*}
		\min_{P\in\cP_m\atop P(0)=I} \bu^*\bldual[auto]{P(\bcA )\circ \bR_0^{(1)}}{P(\bcA )\circ \bR_0^{(1)}}\bu 
		= 
		\min_{P\in\cP_{m-s}\atop P(0) 
		=
		I} \bu^*\bldual[auto]{P(\bcA )\circ \bR_0^{(1)}}{P(\bcA )\circ \bR_0^{(1)}}\bu.
	\end{align*}
	Using \eqref{eq:res_poly_block}, we obtain
	\begin{align*}
		\bu^*\bldual[auto]{\bR^{(s)}_{m}}{\bR^{(s)}_{m}}\bu &= \min_{P\in\cP_s\atop P(0)=I}\bu^*\bldual[auto]{P(\bcA )\circ \bR^{(2)}_{0}}{P(\bcA )\circ \bR^{(2)}_{0}}\bu\\
		&= \min_{P\in\cP_s\atop P(0)=I}\bu^*\bldual[auto]{P(\bcA )\circ \bR^{(1)}_{m}}{P(\bcA )\circ \bR^{(1)}_{m}}\bu\\
		&= \min_{P\in\cP_s\atop P(0)=I}\bu^*\bldual[auto]{P(\bcA )\widehat{P}_{m-s}(\bcA )\circ \bR_0^{(1)}}{P(\bcA )\widehat{P}_{m-s}(\bcA )\circ \bR_0^{(1)}}\bu\\
		&\leq \min_{P\in\cP_m\atop P(0)=I}\bu^*\bldual[auto]{P(\bcA )\circ \bR_0^{(1)}}{P(\bcA )\circ \bR_0^{(1)}}\bu\\
		&= \bu^*\bldual[auto]{\bR^{(1)}_{m}}{\bR^{(1)}_{m}}\bu=\bu^*\bldual[auto]{\bR^{(1)}_{m-s}}{\bR^{(1)}_{m-s}}\bu,
	\end{align*}
	whereby the inequality follows from the fact that we replace the
	polynomial product with a degree $m$ polynomial over which we now
	minimize. This concludes the proof.
\end{proof}

This allows us to define an admissible convergence sequence of residual normalizing quantities in the case of restarted \blgmres.


\subsection{Building a restarted \blgmres factorization}\label{section:build-restarted-block-kryl-factorization}

We take the same approach as with the restarted \gmres case in \Cref{section:build-restarted-kryl-factorization}; we build the factorization using a restarted block Krylov basis,  
\begin{align}
	\bcK^{(1)}
	:=& 
	\begin{bmatrix}
		\bB & A\bB & \cdots & A^{M-1}\bB
	\end{bmatrix}
	\nonumber
	\\
	\bcK^{(2)}
	:=& 
	\begin{bmatrix}
		\bR_M^{(1)} & A\bR_M^{(1)} & \cdots & A^{M-1}\bR_M^{(1)}
	\end{bmatrix}
	\nonumber
	\\
	\nonumber
	\vdots&\\
	\bcK^{(\ell)}
	:=& 
	\begin{bmatrix}
		\bR_M^{(\ell-1)} & A\bR_M^{(\ell-1)} & \cdots & A^{M-1}\bR_M^{(\ell-1)}
	\end{bmatrix}.
	\label{eqn:block-krylov-matrix}
\end{align}
\begin{definition}\label{definition:block-krylov-matrix}
	Similar to \Cref{definition:krylov-matrix}, the columns of $\bcK^{(1)}, \bcK^{(2)}, \ldots, \bcK^{(\ell)}$ taken together form the block span of some subspace of $\S^N$. If all columns of all these block vectors are linearly independent, we call this collection of block vectors a \emph{restarted block Krylov subspace} of $\S^N$, and the matrix that has these block vectors as its columns 
	\begin{align*}
		\widetilde{\bcK}
		:=
		\begin{bmatrix}
			\bcK^{(1)}
			&
			\bcK^{(2}	
			&
			\cdots
			&
			\bcK^{(\ell)}
		\end{bmatrix}
		\in\S^{N\times N}
	\end{align*}
	is called a \emph{restarted block Krylov matrix}.
\end{definition}
A similar result to \Cref{proposition:krylov-matrix-rank} can be proven for the
restarted block Krylov matrix.  We note to the reader the initial no-breakdown
assumption takes on the additional role of ensuring we consider the case that
there is no loss of dimension of the block Krylov subspace. 
\begin{lemma}
	Suppose that the block Krylov matrix satisfies the assumption
	\begin{align*}
		\dim\prn{\range\bcK} = n.
	\end{align*}
	Then it follows that the block restarted Krylov matrix
	\eqref{eqn:block-krylov-matrix} is full rank, \ie,
	\begin{align*}
		\dim\prn{\range \widetilde{\bcK}}
		=
		n
		, 
	\end{align*}
	if and only if there is no stagnation at the end of any of the restart
	cycles. 
\end{lemma}
\begin{proof}
	Suppose that $\bR_M^{(k)}\bu = \bR_{M-1}^{(k)}\bu$.  Since 
	$\range\prn{\bR_{M-1}^{(k)}}\subset\range\prn{\bcK^{(k)}}$, it follows that 
	\begin{math}
		\dim\prn{
			\range\prn{
				\begin{bmatrix}
					\bcK^{(k)}
					&
					\bR_M^{(k)}
				\end{bmatrix}
			}
		}
		\leq
		Mp-1
	\end{math}

	Conversely, since we have assumed no block Arnoldi breakdown for
	\blgmres, $\bcK$ has full column rank, and thus so does $\bcK^{(1)}$.  
	The assumption of no end-of-cycle stagnation means 
	\begin{align*}
		\range\prn{\bR_M^{(2)}}\cap\K_M\prn{A,\bB}
		=
		\emptyset
		\quad 
		\mbox{and}
		\quad 
		\range\prn{A^j\bR_M^{(2)}}\cap\K_M\prn{A,\bB}
		=
		\emptyset
	\end{align*}
	for $1\leq j\leq M-1$.  Thus we conclude that
	\begin{math}
		\dim\prn{
			\range
			\begin{bmatrix}
				\bcK^{(1)}
				&
				\bcK^{(2)}
			\end{bmatrix}
		}
		=
		2Mp.
	\end{math}
	The result follows by induction.
\end{proof}


\subsection{Residual components in the block Arnoldi vectors}\label{section:block-residual-Arnoldi-components}

For the \blgmres iteration, we can (similar to the \gmres case) express the
block residual $\bR_M$ in terms of the block Arnoldi vectors, and this holds
for each restart cycle.  As we have noted, the $jp\times jp$ upper-Hessenberg
matrix $\cH_j$ (the principal submatrix of $\cH$) can be factorized as
\begin{align}
	\cH_j
	=
	\cD_j\cU_j\cC_j(\cD_j\cU_j)^{-1}.
\end{align}
We show how the matrices $\cD_j$ and $\cU_j$ are related to the components of
the block residual and how the components of the residual with respect to the
computed Arnoldi vectors are related to the convergence behavior. Recall from
\cite[Remark 4.2]{KubinovaSoodhalter:2020:1} that the first row of $(\cD \cU
)^{-1}$ satisfies
\begin{align}\label{eq:first_row_form_block}
	\bE_1^T(\cD \cU )^{-1}\bE_1 
	&=
	F_0^{-1}
	\nonumber
	\\
	\bE_1^T(\cD \cU )^{-1}\bE_k 
	&= 
	\sqrt{\bldual[auto]{F_{k-1}}{F_{k-1}}^{-1} -\bldual[auto]{F_{k-2}}{F_{k-2}}^{-1}}^{\,*} Q_{k}, 
	\\
	\mbox{for}
	\quad
	k
	=
	2,\dots,n,
	\nonumber
\end{align}
for some $Q_{k}\in\S$ such that $\abs[auto]{Q_k}= I$.


\begin{lemma}
	The $k$th block residual satisfies  
	\begin{align}
		\bR_k = \bcV ^{(k+1)}\bG_k, \quad \text{where} \quad \bG_k :
		=
		\brak{\bE_1^T\left(\cD ^{(k+1)}\cU ^{(k+1)}\right)^{-1}}^*\bldual[auto]{\bR_k}{\bR_k}.
		\label{eqn:kth-block-residual-coefficients}
	\end{align}
	\label{lemma:block-residual-coefficients}
\end{lemma}


\begin{proof}
	From \eqref{eq:first_row_form} it follows that 
	\begin{align*}
		\bldual[auto]{\bG_k}{\bG_k} 
		=
		\bldual[auto]{\bR_k}{\bR_k}.
	\end{align*}
	Therefore, to prove $\bR_k = \bcV ^{(k+1)}\bG_k$, it suffices to show
	that
	\begin{align*}
		{\bR_k}\bldual[auto]{\bcV ^{(k+1)}\bG_k} 
		=
		\bldual[auto]{\bR_k}{\bR_k}
		.
	\end{align*}
	Using 
	\begin{align*}
		\bR_k 
		=
		\bcV ^{(k+1)}\left(\cI_{k+1} - \underline{\cH }^{(k)}\left(\underline{\cH }^{(k)}\right)^\dagger\right)\bE_1\blnorm{\bR_0}
	\end{align*}
	and
	\begin{align*}
		\underline{\cH }^{(k)} 
		=
		\cD ^{(k+1)}\cU ^{(k+1)}
		\begin{bmatrix}
			\bnull\\ \cI_k
		\end{bmatrix}
		\left(\cD ^{(k)}\cU ^{(k)}\right)^{-1}
	\end{align*}
	along with the first  line of \eqref{eq:first_row_form_block} (since $F_0 = \blnorm{\bR_0}$)
	we obtain
	\begin{align*}
		\bldual[auto]{\bR_k}{\bcV ^{(k+1)}\bG_k} 
		=& 
		\bldual[auto]{\bR_k}{\bR_k}
		\prn{
			\bE_1^T\left(\cD ^{(k+1)}\cU ^{(k+1)}\right)^{-1}\bE_1\blnorm{\bR_0} - \bnull
		}
		\\
		=&
		 \bldual[auto]{\bR_k}{\bR_k}.
	\end{align*}
	
\end{proof}


\subsection{Generating the prescribed block Hessenberg matrices}\label{section:generating-block-Hess}
As with restarted \gmres, we characterize the case with no end-of-cycle stagnation first and build upon that to accommodate stagnation.


\subsubsection{No end-of-cycle stagnation in any direction}
If there is no end-of-cycle stagnation with respect to any direction, we can
generalize the results of \Cref{theorem:restart-gmres-nostag-assignment} to
allow for assigning arbitrary admissible residual convergence for each cycle while
specifying eigenvalues and Ritz values independently, subject to the
restrictions imposed by the assignment of stagnation at a particular iteration.
If a stagnation is assigned at some iteration with respect to any direction, it follows 
that at least one Ritz value must be zero at that iteration \cite{Soodhalter:2017:1}.

We verify that the admissible convergence for restarted \gmres with no end-of-cycle stagnation generalizes to the restarted \blgmres.
\begin{lemma}
	Let $F,G\in\S^+$ be two normalizing quantities.  Then $F$ and $G$ are 
	strictly ordered (\ie $F<G$ or $G<F$)
	if and only if there is no direction $\bu$ such that
	\begin{align}
		\bu^\ast F^\ast F\bu = \bu^\ast G^\ast G\bu.
		\label{eqn:u-test-equal}
	\end{align}
\end{lemma}
\begin{proof}
	This result is a direct consequence of the fact that we are assuming
	that $F^\ast F-G^\ast G$ is definite, from which it immediately follows
	that \eqref{eqn:u-test-equal} cannot hold.
%
\end{proof}
\begin{theorem}\label{theorem:restart-blgmres-nostag-assignment}
	Let $\curly{\curly{F_j^{(k)}}_{j=0}^{M-1}}_{k=1}^{\ell}$ be a sequence of positive normalizing quantities (\ie, $F_j^{(k)}\in\S^+$ for all $j,k$) satisfying
	\begin{align*}
		F_0^{(1)}
		\geq
		F_1^{(1)}
		\geq
		\cdots 
		\geq
		F_{M-1}^{(1)}
		>
		F_{0}^{(2)}
		\geq
		\cdots 
		\geq
		F_M^{(\ell-1)}
		> 
		F_0^{(\ell)}
		\geq
		\cdots 
		\geq
		F_{M-1}^{(\ell)}>0,
	\end{align*}
	with strictly decreasing block residual norms at the transitions
	between restart cycles.  Then there
	exists a matrix $A$ and a right-hand side $\bB$ such that restarted
	\blgmres with cycle length $M$ applied to \eqref{eqn:AXB} yields
	residuals satisfying $\blnorm{\bR_j^{(k)}}=F_j^{(k)}$ for all
	$j=0,1,\ldots, M-1$ and $k=0,1,\ldots, M-1$. 
\end{theorem}
\begin{proof}
	Following the proof of \Cref{theorem:restart-gmres-nostag-assignment},
	we can \Wlog consider the case of constructing an appropriate block
	Hessenberg matrix with starting vector being the first block standard
	basis vector $\bE_1\in\S^N$ and then generalize by applying a basis
	transformation, as for non-block restarted \gmres. This works because
	for any block Hessenberg matrix $\cH\in\S^{N\times N}$, the $i$th
	block standard basis vector $\bE_i$ ($1\leq i\leq N-1$) satisfies 
	\begin{align}
		\cH\bE_i 
		\in
		\blspan{
			\braces{
				\bE_1,
				\bE_2,
				\ldots,
				\bE_i,
				\bE_{i+1}
			}
		}
		.
		\label{eqn:block-hessenberg-arnoldi-vectors}
	\end{align}
	This means that the (non-restarted) block Arnoldi basis for $\cH$
	generated by $\bE_1$ is the block standard basis vectors.

	Since we enforce $F_{M-1}^{(j)} > F_0^{(j+1)}>0$ for each cycle $j$, we use the theory from \cite{KubinovaSoodhalter:2020:1} to construct the block upper-Hessenberg matrix for each cycle,
	\begin{align*}
		\cH^{(k)}
		=
		\cD^{(k)}\cU^{(k)}\cC^{(k)}\left(\cD^{(k)}\cU^{(k)}\right)^{-1}\in\S^{(M+1)\times (M+1)}, 
		\quad 
		k
		=
		1,\ldots,\ell,
	\end{align*}
	such that \blgmres applied to $(\cH^{(k)},\bE_{k\ell+1} F_0^{(k)})$
	yields the prescribed residual convergence for the $k$th cycle, with
	the prescribed Ritz values at each iteration in the cycle and with
	$\cH^{(k)}$ having the prescribed solvents.  It follows from
	\eqref{eqn:block-hessenberg-arnoldi-vectors} that the block Arnoldi
	vectors can be shown to be subsets of standard basis vectors of $\S^N$;
	\ie
	\begin{align*}
		\bcV^{(1)}
		:=&
		\begin{bmatrix}
			\bE_1
			&
			\bE_2
			&
			\cdots
			&
			\bE_M
		\end{bmatrix}
		\\
		\bcV^{(k+1)}
		:=&
		\begin{bmatrix}
			\bcV^{(k)}
			&
			\bE_{kM+2}
			&
			\bE_{kM+3}
			&
			\cdots\bE_{(k+1)M}
		\end{bmatrix}
		\begin{bmatrix}
			\bG^{(k)}\blnorm{\bG^{(k)}}
			&
			\\
			&
			I_{M-1}
		\end{bmatrix},
		\\
		\mbox{for}
		\quad
		k
		=&
		1,2,\ldots, \ell-1,
	\end{align*}
	where $\bG^{(k)}=\brak{\bE_1^T\left(\bcD ^{(k+1)}\bcU
	^{(k+1)}\right)^{-1}}^*\bldual[auto]{F_0^{(k+1)}}{F_0^{(k+1)}}$, from
	\eqref{eqn:kth-block-residual-coefficients} and that the final block
	residual in cycle $k$,
	$\bR_M^{(k)}\in\blspan\braces{\bE_1,\ldots,\bE_{k\ell+1}}$.  From
	\eqref{eqn:block-hessenberg-arnoldi-vectors}, this means
	$\cH\bR_M^{(k)}\in\blspan\braces{\bE_1,\ldots,\bE_{k\ell+1},\bE_{k\ell+2}}$,
	for any block Hessenberg matrix $\bH$. Thus the first new Arnoldi
	vector generated after an iteration of block Arnoldi must be
	$\bE_{k\ell+2}$. The matrices $\bcV^{(k)}$ have orthonormal columns for
	all $k$, and since there is no end-of-cycle stagnation, the matrix 
	\begin{math}
		\widetilde{\bcV}
		=
		\begin{bmatrix}
			\bcV^{(1)}
			&
			\bcV^{(2)}
			&
			\cdots
			&
			\bcV^{(\ell)}
		\end{bmatrix}
	\end{math}
	has full rank.  

	We set $\bB = \bE_1 F_0^{(1)}$ and define $A$ via its action on a set of linearly independent vectors; \ie,
	\begin{align}
		A\bcV^{(k)}
		=&
		\begin{bmatrix}
			\bcV^{(k)}
			&
			\bE_{kM+1}
		\end{bmatrix}
		\underline{\cH^{(k)}}
		\quad
		\mbox{for}
		\quad
		k-1=
		1, 2,\ldots, \ell-1
		\nonumber
		\\
		A\bcV^{(\ell)}
		=&
		\begin{bmatrix}
			\bcV^{(\ell)}
			&
			\bnull
		\end{bmatrix}
		\underline{
			\cH
		}^{(\ell)}
		=
		\bcV^{(\ell)}
		\begin{bmatrix}
			\cI_M
			&
			\bnull
		\end{bmatrix}
		=
		\underline{
			\cH
		}^{(\ell)}
		\bcV^{(\ell)}\cH^{(\ell)}_{M\times M}.
		\label{eqn:block-cycle-arnoldi-relations}
	\end{align}
	where $\cH^{(\ell)}_{M\times M}$ contains the first $M$ block rows of $\underline{\cH}^{(\ell)}$.
	The construction of the Hessenberg matrices and of the action of $A$ yields the prescribed block residual norm convergence sequence.  Furthermore, the cycles are further tied together via \eqref{eqn:kth-block-residual-coefficients},
	\begin{align}
		\bcV^{(k+1)}\bE_1
		=
		\begin{bmatrix}
			\bcV^{(k)}
			&
			\bE_{kM+1}
		\end{bmatrix}
		\bG^{(k)}\blnorm{\bG^{(k)}}^{-1}
		=
		\bR^{(k)}_M\blnorm{\bR^{(k)}_M}^{-1}.
		\label{eqn:bl-arnoldi-cycles-tied-together}
	\end{align}
\end{proof}

We note that as in the case of restarted \gmres, we can rewrite \eqref{eqn:block-cycle-arnoldi-relations} as 
\begin{align}
	A\bcV^{(k)}
	=
	\begin{bmatrix}
		\bcV^{(k)}
		&
		\bcV^{(k+1)}\bE_1
	\end{bmatrix}
	\underbrace{
	\begin{bmatrix}
		\begin{matrix}
			\cI_m
			\\
			\bnull
		\end{matrix}
		&
		\bG^{(k)}\blnorm[auto]{\bG^{(k)}}^{-1}
	\end{bmatrix}
	\underline{\cH^{(k)}}
	}_{\eqcolon \underline{\widetilde{\cH}^{(k)}}},
\end{align}
allowing us to define $A$ fully by its action on all the restarted block Arnoldi vectors 
\begin{align}
	A\widetilde{\bcV}
	=
	\widetilde{\bcV}
	\underbrace{
		\begin{bmatrix}
			\underline{\widetilde{\cH}^{(1)}}
			\\
			&
			\underline{\widetilde{\cH}^{(2)}}
			\\
			&
			&
			\ddots
			\\
			&
			&
			&
			\underline{\widetilde{\cH}^{(\ell-1)}}
			\\
			&
			&
			&
			&
			\cH^{(\ell)}_{M\times M}
		\end{bmatrix}
	}_
	{
		=: 
		\widetilde{\cH}
	}
	\label{eqn:full-block-restarted-Arnoldi-relation}
\end{align}
This construction allows us to see, as in
\Cref{corollary:eigs_of_A,corollary:ritzvals_of_A}, that the eigenvalues and
Ritz values for each iteration of each cycle are determined by the block
Hessenberg blocks of $\widetilde{\cH}$ and can thus all be assigned,
independent of the block residual normalizing quantity behavior.  The proofs
follow those of the single-vector case and are thus omitted.
\begin{corollary}
	The eigenvalues and Ritz values for each individual block Hessenberg
	matrix for each cycle can be assigned arbitrarily, unrelated to the
	sequence of block residual normalizing quantities, in the case that
	there is no stagnation \wrt to any direction specified.  In the case
	that a stagnation is specified at iteration $j$ of cycle $k$, then at
	least one Ritz value of $\cH_j^{(k)}$ must be zero. 
\end{corollary}

Lastly, we observe that as with the single-vector restarted \gmres case, for
each cycle of restarted \blgmres, we can assign the block Arnoldi vector for
each cycle, subject to the constraint that the cycles are tied together as in
\eqref{eqn:bl-arnoldi-cycles-tied-together}.

\begin{corollary}
	The results from \Cref{theorem:restart-blgmres-nostag-assignment} are
	valid for any basis of mutually independent sets of block vectors,
	whereby for each $k$, $\bcW^{(k)}$ has orthonormal columns and for each
	$k$, their first columns satisfy
	\begin{align*}
		\bcW^{(k+1)}\bE_1
		=
		\begin{bmatrix}
			\bcW^{(k)}
			&
			\widetilde{\bW}_k
		\end{bmatrix}
		\bG^{(k)}\blnorm[auto]{\bG^{(k)}}^{-1},
	\end{align*}
	where $\widetilde{\bW}_k$ is any normalized block vector orthogonal to
	the columns of $\bcW^{(k)}$.
\end{corollary}

\begin{proof}
	Consider that we have prescribed the restarted \blgmres behavior for a
	block upper-Hessenberg system of the form 
	\begin{align*}
		\widetilde{\cH}\bX 
		=
		\bE_1F_0^{(1)}
		.
	\end{align*}
	Consider
	\begin{math}
		\widetilde{\bcW}
		=
		\begin{bmatrix}
			\bcW^{(1)}
			&
			\bcW^{(2)}
			&
			\cdots
			&
			\bcW^{(\ell)}
		\end{bmatrix}
		\in 
		\S^{N\times N}
	\end{math}
	wherein each $\bcW^{(i)}\in\S^{N\times M}$ has orthonormal 
	columns but $\bcW^{(i)}$ and $\bcW^{(j)}$ do not necessarily 
	have mutually orthogonal columns for $i\neq j$, such that 
	$\widetilde{\bcW}$ is invertible.

	Setting $A = \widetilde{\bcW}\widetilde{\cH}\widetilde{\bcW}^{-1}$, 
	we observe that 
	\begin{align*}
		A\widetilde{\bcW}
		=
		\widetilde{\bcW}\widetilde{\cH}\widetilde{\bcW}^{-1}
		\widetilde{\bcW}
		=
		\widetilde{\bcW}\widetilde{\cH}
		.
	\end{align*}
	Thus, $\widetilde{\bcW}$ is a restarted Arnoldi basis for $A$.  Furthermore,
	\begin{align*}
		\widetilde{\cH}\bX 
		=&
		\bE_1F_0^{(1)}	
		\\
		\iff
		\widetilde{\bcW}
		\widetilde{\cH}
		\widetilde{\bcW}^{-1}
		\widetilde{\bcW}
		\bX 
		=&
		\widetilde{\bcW}\bE_1F_0^{(1)}
		\\
		\iff 
		A \widehat{\bX}
		=& 
		\bV_1^{(1)}
		F_0^{(1)}
		,
	\end{align*}
	where $\widehat{\bX} = \widetilde{\bcW}\bX$. Thus, for an arbitrarily
	chosen restarted block Krylov basis $\widetilde{\bcW}$, producing the
	assigned restarted \blgmres behavior, with the assigned eigenvalues and
	Ritz values.
\end{proof}

Directly analogous to the case in single-vector restarted \gmres, there is the
same implicit choice made in the last iteration that allows us to assign all
the eigenvalues of $A$ and the Ritz values at each iteration of \blgmres.  A
more general expression of the block Arnoldi iteration at the last cycle would
be 
\begin{align*}
	A\bcV^{(\ell)}
	=
	\begin{bmatrix}
		\bcV^{(\ell)}
		&
		\bU
	\end{bmatrix}
	\underline{\cH^{(\ell)}},
\end{align*}
but with $\bU=\widetilde{\bcV}\widehat{\bC}$ being a linear combination of the basis of restarted block Arnoldi vectors with $\widehat{\bC}\in\S^N$ being an arbitrarily chosen vector.  This vector can be chosen freely to influence the block residual normalizing quantity $\bR_M^{(\ell)}$ at the expense of no longer being able to exactly specify the eigenvalues of $A$ or the Ritz values of each cycle. The choice $\widehat{\bC}=\bnull$ corresponds to the choice allowing them to be specified unambiguously.  
\begin{lemma}
	Let $\bU\neq \bnull$.  This corresponds to the rank-one update of the full block restarted Arnoldi relation \eqref{eqn:full-block-restarted-Arnoldi-relation}
	\begin{align}
		A\widetilde{\bcV}
		=
		\widetilde{\bcV}
		\underbrace{
			\prn{
				\widetilde{\cH} + \bC\bE_N^T
			}
		}_{\eqcolon \widetilde{\cH}\prn{\bC}}
	\end{align}
	where $\bC = \widehat{\bC}H_{M+1,M}^{(\ell)}$, and $H_{M+1,M}^{(\ell)}$
	is the $(M+1,M)$th block entry of $\underline{\cH^{(\ell)}}$.
\end{lemma}
\begin{proof}
	The proof is analogous to that of \Cref{lemma:rank-one-update-Hessenberg-laststep}.
\end{proof}
\begin{corollary}
	Let $\bx_M^{(\ell)} = \widetilde{\bcV}\bY$, $\bY\in\S^N$, be the final
	restarted \blgmres approximation for the $\ell$th cycle.  Then the
	residual $\bR_M^{(\ell)}= \bB - A \bX_M^{(\ell)}$ admits the expression
	\begin{align*}
		\bR_M^{(\ell)}
		=
		\widetilde{\bcV}
		\prn{
			\bE_1\blnorm{\bB}
			-
			\widetilde{\cH}\bY
			-
			\bC Y_N
		}
	\end{align*}
	where $Y_N$ is the $N$th block component of $\bY$.
\end{corollary}
\begin{corollary}
	The proof is analogous to that of \Cref{corollary:final-residual-expression}.
\end{corollary}

\begin{remark}
	We note that non-zero choices of $\bC$ do not effect any other
	approximation or the previously prescribed residual normalizing
	quantities.  As in the non-block case, it does effect the eigenvalues
	and Ritz values, and it does effect the norm and other properties of
	$\blnorm{\bR_M^{(\ell)}}$.  However, these effects are not easily
	represented as a function of $\bC$, since the columns of $\widetilde{\bcV}$ 
	do not generally form an orthonormal basis.
\end{remark}

\begin{remark}
	We also note that as in the non-block case, if we allow for iterations
	of partial or full stagnation at the end of a restart cycle, this will
	lead to linear dependence among the columns of $\widetilde{\bcV}$.  In
	order to obtain a full factorization that allows us to uniquely
	construct $A$, additional iterations for addition cycles of restarted
	\blgmres will need to be specified, until a full basis for $\S^N$ has
	been generated among the restarted block Arnoldi vectors.
\end{remark}

\section{Conclusions and future work}\label{section:conclusions}

We have extended results from \cite{DuintjerTebbensMeurant:2019:1} to the
restarted \blgmres setting.  To do this, we extended some results from
\cite{DuintjerTebbensMeurant:2019:1} for the non-block restarted \gmres case,
and we also presented new proofs of certain results to facilitate extending
them to the block case.  We note that aspects of our work again demonstrate the
utility of treating the block Krylov iterative methods as iterations over the
right vector space $\S^N$ with scalars from the $^\ast$-algebra $\S=\C^{p\times
p}$, following \cite{KubinovaSoodhalter:2020:1}.  

We note that as in \cite{KubinovaSoodhalter:2020:1}, we are not yet able to
accommodate the occurrence of block Arnoldi breakdown (\ie linear dependence of
the block Arnoldi vectors) in this setup.  The next step will be to understand
how to also prescribe the occurrence of breakdown and to understand what
admissible convergence behavior is in this case, for both restarted and
non-restarted \blgmres.

\section*{Acknowledgements}
The author gratefully acknowledges Marie Kubínová for her significant
contributions to this research, begun during her postdoctoral time before a
transition to work in industry. Her extensive efforts in the initial
development and analysis were essential to the completion of this work.
The author would also like to thank Gerard Meurant for his helpful comments
and pointers to two references.

\appendix

\section{Characterizing the \gmres -- \fom relationship}\label{appendix:fom-gmres-relationship}

There is a sibling method to \gmres called the Full Orthogonalization Method (\fom) that is defined by the nonoptimal residual
constraint
\begin{align*}
	\mbox{select}
	\ 
	\bx_j
	\in
	\cK_j(A,\bb)
	\ 
	\mbox{such that}
	\ 
	\bb - A\bx	
	\perp
	\cK_j(A,\bb).
\end{align*}
This is equivalent to computing $\by_j =
H_j^{-1}\paren[auto]{(}{)}{\norm[auto]{\br_0}\be_1}$.  The \gmres and \fom
iterates have a well-documented \emph{peak-plateau} relationship (\cf
\Cref{corollary:residual-proportion}; see, e.g.,
\cite{Brown:1991:1,Walker:1995:1}) that we use in our analysis.  When needed,
we differentiate quantities associated to \gmres and \fom via superscripts for
the \fom quantities; \eg, $\bx_j^{(F)} = V_j\by_j^{(F)}$.

In \Cref{section:prescribing-restarted-gmres} we characterize matrices and right-hand sides with prescribed convergence behavior
of restarted Arnoldi and \gmres. To do this, we establish a result that is a direct corollary of the peak-plateau
relationship of \gmres and \fom \cite{Brown:1991:1,Walker:1995:1}.
\begin{corollary}\label{corollary:residual-proportion}
The components of the \gmres residual in the direction of a particular Arnoldi vector are inversely proportional to the corresponding \fom residual in the following sense
\begin{align*}
	\dual[auto]{\br_m}{\br_m}^{-1}\dual[auto]{\bv_{i+1}^*\br_m}{\bv_{i+1}^*\br_m}\dual[auto]{\br_m}{\br_m}^{-1} 
	= 
	\dual[auto]{\br_i^{(F)}}{\br_i^{(F)}}^{-1}, \quad i
	=
	0,\ldots,k,
\end{align*}	
which can be rewritten as 
\begin{align*}
	\frac{|\bv_{i+1}^*\br_m|}{\norm[auto]{\br_m}} 
	=
	\frac{\norm[auto]{\br_m}}{\norm[auto]{\br^{(F)}_i}}.
\end{align*}
\end{corollary}

The following corollary (to \Cref{lemma:block-residual-coefficients}) follows then directly from the peak plateau relation.
\begin{corollary}
The components of the \blgmres residual in the direction of a particular Arnoldi vector are inversely proportional to the corresponding \blfom residual in the following sense
\begin{align*}
	&
	\bldual[auto]{\bR^G_k}{\bR^G_k}^{-1}
	\bldual[auto]{\bV_{j+1}^*\bR^G_k}{\bV_{j+1}^*\bR^G_k}
	\bldual[auto]{\bR^G_k}{\bR^G_k}^{-1} 
	= 
	\bldual[auto]{\bR_j^{(F)}}{\bR_j^{(F)}}^{-1}, 
\end{align*}
for all $j = 0,\ldots,k$.
\end{corollary}

Partial stagnation can then be related to the characteristic peak-plateau behavior.


\begin{lemma}
	Stagnation with respect to $\bu$ appears if and only if the difference
	\begin{align*}
		\bldual[auto]{\bR^{(k)}_{m}}{\bR^{(k)}_{m}}^{-1} 
		- 
		\bldual[auto]{\bR^{(k)}_{m-s}}{\bR^{(k)}_{m-s}}^{-1}
	\end{align*}
	is singular (and
	therefore the \blfom solution does not exist).
\end{lemma}


\begin{proof}
	Assume without loss of generality $k=1$. Block stagnation with respect to $\bu$ can be recast as 
	\begin{align*}
		\bu^*\left(\bldual[auto]{\bR^{(1)}_{m-s}}{\bR^{(1)}_{m-s}} 
		-
		\bldual[auto]{\bR^{(1)}_{m}}{\bR^{(1)}_{m}}\right)\bu 
		=
		0.
	\end{align*}
	Since 
	\begin{math}
		\bldual[auto]{\bR^{(1)}_{m-s}}{\bR^{(1)}_{m-s}} 
		-
		\bldual[auto]{\bR^{(1)}_{m}}{\bR^{(1)}_{m}}
	\end{math}
	is positive semidefinite, we have
	\begin{align*}
		\prn{
			\bldual[auto]{\bR^{(1)}_{m-s}}{\bR^{(1)}_{m-s}} 
			-
			\bldual[auto]{\bR^{(1)}_{m}}{\bR^{(1)}_{m}}
		}
		\bu 
		=& 
		\bnull
		\\
		\bldual[auto]{\bR^{(1)}_{m-s}}{\bR^{(1)}_{m-s}}\bu 
		=&  
		\bldual[auto]{\bR^{(1)}_{m}}{\bR^{(1)}_{m}}\bu
		\\
		\bu 
		=&  
		\bldual[auto]{\bR^{(1)}_{m-s}}{\bR^{(1)}_{m-s}}^{-1}
		\bldual[auto]{\bR^{(1)}_{m}}{\bR^{(1)}_{m}}\bu
		\\
		\bldual[auto]{\bR^{(1)}_{m}}{\bR^{(1)}_{m}}^{-1}
		\underbrace{
			\bldual[auto]{\bR^{(1)}_{m}}{\bR^{(1)}_{m}}\bu
		}_{\bu} 
		=&  
		\bldual[auto]{\bR^{(1)}_{m-s}}{\bR^{(1)}_{m-s}}^{-1}
		\underbrace{
			\bldual[auto]{\bR^{(1)}_{m}}{\bR^{(1)}_{m}}\bu
		}_{\bu}
		\\
		\prn{
			\bldual[auto]{\bR^{(1)}_{m}}{\bR^{(1)}_{m}}^{-1}
			-
			\bldual[auto]{\bR^{(1)}_{m-s}}{\bR^{(1)}_{m-s}}^{-1}
		}
		\bu
		=&
		\bnull.
	\end{align*}
	Therefore 
	\begin{math}
		\bldual[auto]{\bR^{(k)}_{m}}{\bR^{(k)}_{m}}^{-1} 
		- 
		\bldual[auto]{\bR^{(k)}_{m-s}}{\bR^{(k)}_{m-s}}^{-1}
	\end{math}
	is singular. The non-existence of the \blfom solution follows from the peak-plateau relation.
\end{proof}

\printbibliography

\end{document}